\documentclass[a4paper,11pt]{article}
\usepackage[latin1]{inputenc}
\usepackage[english]{babel}
\usepackage{amsmath}
\usepackage{amsfonts}
\usepackage{amssymb}
\usepackage{epsfig}
\usepackage{amsopn}
\usepackage{amsthm}
\usepackage{color}
\usepackage{graphicx}
\usepackage{subfigure}
\usepackage{enumerate}
\newtheorem{theorem}{Theorem}[section]

\newtheorem{lemma}[theorem]{Lemma}
\newtheorem{proposition}[theorem]{Proposition}
\newtheorem{definition}[theorem]{Definition}

\newtheorem*{theorem*}{Theorem}
\newtheorem*{lemma*}{Lemma}
\newtheorem*{remark*}{Remark}
\newtheorem*{definition*}{Definition}
\newtheorem*{proposition*}{Proposition}
\newtheorem*{corollary*}{Corollary}
\numberwithin{equation}{section}
\newcommand{\real}{\mathbb{R}}

\let\ced=\c         

\def\qed{\,\unskip\kern 6pt \penalty 500
\raise -2pt\hbox{\vrule \vbox to8pt{\hrule width 6pt
\vfill\hrule}\vrule}\par}
\definecolor{darkblue}{rgb}{0.05, .05, .65}
\definecolor{darkgreen}{rgb}{0.1, .65, .1}
\definecolor{darkred}{rgb}{0.8,0,0}
\newcommand{\beqn}{\begin{equation}}
\newcommand{\eeqn}{\end{equation}}
\newcommand{\bear}{\begin{eqnarray}}
\newcommand{\eear}{\end{eqnarray}}
\newcommand{\bean}{\begin{eqnarray*}}
\newcommand{\eean}{\end{eqnarray*}}
\begin{document}

\title{\huge \bf Self-similar blow-up profiles for a reaction-diffusion equation with strong weighted reaction}

\author{
\Large Razvan Gabriel Iagar\,\footnote{Departamento de Matem\'{a}tica
Aplicada, Ciencia e Ingenieria de los Materiales y Tecnologia
Electr\'onica, Universidad Rey Juan Carlos, M\'{o}stoles,
28933, Madrid, Spain, \textit{e-mail:} razvan.iagar@urjc.es},
\\[4pt] \Large Ariel S\'{a}nchez,\footnote{Departamento de Matem\'{a}tica
Aplicada, Ciencia e Ingenieria de los Materiales y Tecnologia
Electr\'onica, Universidad Rey Juan Carlos, M\'{o}stoles,
28933, Madrid, Spain, \textit{e-mail:} ariel.sanchez@urjc.es}\\
[4pt] }
\date{}
\maketitle

\begin{abstract}
We study the self-similar blow-up profiles associated to the following second order reaction-diffusion equation with strong weighted reaction and unbounded weight:
$$
\partial_tu=\partial_{xx}(u^m) + |x|^{\sigma}u^p,
$$
posed for $x\in\real$, $t\geq0$, where $m>1$, $0<p<1$ and $\sigma>2(1-p)/(m-1)$. As a first outcome, we show that finite time blow-up solutions in self-similar form exist for $m+p>2$ and $\sigma$ in the considered range, a fact that is completely new: in the already studied reaction-diffusion equation without weights there is no finite time blow-up when $p<1$. We moreover prove that, if the condition $m+p>2$ is fulfilled, all the self-similar blow-up profiles are compactly supported and there exist \emph{two different interface behaviors} for solutions of the equation, corresponding to two different interface equations. We classify the self-similar blow-up profiles having both types of interfaces and show that in some cases \emph{global blow-up} occurs, and in some other cases finite time blow-up occurs \emph{only at space infinity}. We also show that there is no self-similar solution if $m+p<2$, while the critical range $m+p=2$ with $\sigma>2$ is postponed to a different work due to significant technical differences.
\end{abstract}

\


\smallskip


\section{Introduction}

The goal of this paper is to study and classify the self-similar blow-up profiles for the following reaction-diffusion equation with weighted reaction
\begin{equation}\label{eq1}
u_t=(u^m)_{xx}+|x|^{\sigma}u^p, \qquad u=u(x,t), \quad
(x,t)\in\real\times(0,T),
\end{equation}
in the following range of exponents
\begin{equation}\label{range.exp}
m>1, \ \quad 0<p<1, \quad \sigma>\frac{2(1-p)}{m-1},
\end{equation}
where, as usual, the subscript notation in \eqref{eq1} indicates partial derivative with respect to the time or space variable. By finite time blow-up we understand the situation when a solution which was bounded before, becomes unbounded at time $T\in(0,\infty)$. More precisely, we say that a solution
$u$ to \eqref{eq1} blows up in finite time if there exists $T\in(0,\infty)$ such that $u(T)\not\in L^{\infty}(\real)$, but $u(t)\in L^{\infty}(\real)$ for any $t\in(0,T)$. The smallest time $T<\infty$ satisfying this property is known as the blow-up time of $u$. Here and in the sequel, we denote by $u(t)$ the map $x\mapsto u(x,t)$ for a fixed time $t\in[0,T]$. The present work is a part of a larger project developed by the authors having the aim to understand the blow-up behavior of solutions to reaction-diffusion equations with weighted reaction and unbounded weights.

The reaction-diffusion equation
\begin{equation}\label{eq1.hom}
u_t=(u^m)_{xx}+u^p
\end{equation}
has been considered since long and its blow-up behavior in the range $p>1$ is nowadays well understood, at least in one space dimension. Good surveys of the classical results on finite time blow-up for \eqref{eq1.hom} with either $m=1$, or $m>1$ but $p>1$ can be found in the books \cite{QS} and \cite{S4}. However, in the present work we consider exponents $p\in(0,1)$, a case in which it is known that finite time blow-up does not occur for bounded and compactly supported initial conditions. Eq. \eqref{eq1.hom} for exponents $p\in(0,1)$ has been considered in a series of papers by de Pablo and V\'azquez \cite{dPV90, dPV91, dPV92} where the rather complex but very interesting qualitative theory is developed. In this sequence of works it is shown that the Cauchy problem associated to Eq. \eqref{eq1.hom} is generally ill-posed, as uniqueness of solutions is lacking. More precisely, local existence of solutions is established in suitable functional spaces and it is moreover shown that all the solutions (if more than one) having the same initial condition can be ordered between a minimal solution $\underline{u}$ and a maximal solution $\overline{u}$ obtained as a limit process \cite{dPV91}. Concerning deeper qualitative properties of solutions such as uniqueness, finite or infinite speed of propagation, interface equation, it is shown in \cite{dPV90} that many of these properties depend strongly on the sign of $m+p-2$, for example

$\bullet$ if $m+p-2\geq0$, finite speed of propagation of compactly supported solutions occurs and given a bounded initial condition $u_0$, it is shown that uniqueness of solutions holds true \emph{if and only if $u_0(x)>0$} for any $x\in\real$. Indeed, the authors of \cite{dPV90} prove that the maximal solution $\overline{u}$ is always positive, while the minimal solution $\underline{u}$ has always compact support if $u_0$ is itself compactly supported. Thus, at least two solutions are obtained, while for positive data $u_0$ uniqueness is established.

$\bullet$ if $m+p-2<0$, infinite speed of propagation is established in \cite{dPV90}: for any data $u_0$ such that $u_0\not\equiv0$, local solutions become strictly positive $u(x,t)>0$ for any $t>0$, thus uniqueness then holds true along the lines of the previous case.

The non-uniqueness of solutions to Eq. \eqref{eq1.hom} has been further investigated in \cite{dPV92}, and a classification of all the possible solutions starting from a fixed initial condition is given. Moreover, the large time behavior of solutions is addressed in \cite{dP94}, and in all these works the self-similar solutions of the equation, of the form
$$
u(x,t)=t^{-\alpha}f(xt^{-\beta})
$$
with suitable exponents $\alpha$, $\beta$ and profiles $f$ play a significant role both as subjects for the comparison principle and as patterns that the solutions approach for large times \cite{dP94}. This proves the importance of having a good knowledge of the self-similar solutions to Eq. \eqref{eq1}, such solutions are expected to give the patterns of the whole dynamics of the equation. Moreover, they are often used also for comparison with other solutions whose bounds are established in this way.

Concerning the reaction-diffusion equations with weighted reaction terms, a number of works are devoted to their qualitative theory and focus on the existence of the Fujita exponents (that is, exponents $p_*$ such that, for $p<p_*$ any solution to Eq. \eqref{eq1} blows up in finite time) and, above this exponent, on giving further conditions on the initial data $u_0$ for finite time blow-up to take place, or on the contrary, smallness conditions insuring that the solutions to Eq. \eqref{eq1} are global. We recall here, in the semilinear case, the works by Pinsky \cite{Pi97, Pi98} and for the slow diffusion $m>1$, $p>m$, the very general paper by Suzuki \cite{Su02} establishing conditions on the tail of $u_0(x)$ as $|x|\to\infty$ for the blow-up to take place. Andreucci and Tedeev \cite{AT05} establish the blow-up rates for $m>1$, $p>m$ and suitable range of $\sigma>0$, even in the more general case of the doubly nonlinear equation. More recent papers deal with more general cases of unbounded weights, either pure positive powers or pure negative powers (that are unbounded at the origin), or even studying finite time blow-up for equations with two weights, one on the reaction term and another one on $\partial_tu$, such as for example \cite{WZ06, MT08, MTS12}. When the reaction is weighted with a pure power term such as $|x|^{\sigma}$, which vanishes at $x=0$, another natural question is whether $x=0$ (and more generally the zeros of the weight in the case of a general weight $V(x)$) can be a blow-up point. This has been studied in \cite{GLS10, GLS13, GS11, GS18}, focusing on the case of the homogeneous Dirichlet problem in a bounded domain.

Recently the authors started a long term project of understanding the dynamics of Eq. \eqref{eq1} in different cases of $m$, $p$ and $\sigma$, with the aim of answering some finer questions concerning the finite time blow-up: classifying the blow-up sets, obtaining blow-up rates and if possible, establishing the patterns of general solutions near the blow-up time. Taking into account the relevance of the self-similar solutions for these questions, we focused on classifying the possible blow-up patterns for Eq. \eqref{eq1}, obtaining some interesting and completely new types of profiles (whose existence depends on the magnitude of $\sigma$) that do not exist in the non-weighted case. We also show in \cite{IS19a} that for $p=1$ but $\sigma>0$ finite time blow-up produces, a fact that is not true with $\sigma=0$ (that is, without a weight). In another recent work \cite{IS20} we show that for the critical case $p=m>1$ there exist multiple blow-up profiles if $\sigma>0$ is sufficiently small but all these profiles cease to exist when $\sigma$ increases, a fact that has to be further understood (as in that case, the blow-up phenomenon is no longer possible to follow a global in space self-similar pattern). Finally, in \cite{IS19a, IS19b} a study of self-similar profiles is performed for $1\leq p<m$ showing that the profiles and their blow-up sets strongly differ with respect to $\sigma$: finite time blow-up occurs globally for $\sigma>0$ small, while the blow-up set of the profiles is shown to be only the space infinity when $\sigma>0$ increases, due to the strength of the power $|x|^{\sigma}$ when $|x|$ is very large. The present work is aimed to continue this study, for the very interesting case when $0<p<1$ but $\sigma>0$ is large enough in order to force solutions to blow up in finite time. The general qualitative theory of a very similar reaction-diffusion equation to our Eq. \eqref{eq1} with $p<1$ will be developed in a companion paper \cite{IMS20}, where the results of the present work are strongly used.

\medskip

\noindent \textbf{Main results.} As we have explained above, this paper deals with the self-similar blow-up profiles for Eq. \eqref{eq1}, in the range of exponents \eqref{range.exp}. It is a well established fact that the self-similar solutions to \eqref{eq1} contain significant information on the qualitative properties of general solutions: indeed, on the one hand they are expected to give the "optimal" behavior in a priori estimates for general solutions and on the other hand they are the patterns that generic solutions approach asymptotically (either as $t\to\infty$ in the case of global solutions, or as $t\to T$ if finite time blow-up occurs). Thus, knowing how the self-similar profiles behave is an information of utmost importance in the study of nonlinear diffusion and reaction-diffusion equations. In the case of Eq. \eqref{eq1} and exponents as in \eqref{range.exp}, our classification of self-similar solutions shows in particular that we are in a range where solutions are expected to blow up in finite time, as self-similar solutions do. To be more precise, we look for \emph{backward self-similar solutions} in the form
\begin{equation}\label{SSform}
u(x,t)=(T-t)^{-\alpha}f(\xi), \quad \xi=|x|(T-t)^{\beta},
\end{equation}
where $T\in(0,\infty)$ is the finite blow-up time and $\alpha>0$, $\beta\in\real$ exponents to be determined. Replacing the form \eqref{SSform} in \eqref{eq1}, we readily find that
\begin{equation}\label{SSexp}
\alpha=\frac{\sigma+2}{\sigma(m-1)+2(p-1)}>0, \quad \beta=\frac{m-p}{\sigma(m-1)+2(p-1)}>0
\end{equation}
and the self-similar profile $f$ is a solution to the non-autonomous differential equation
\begin{equation}\label{SSODE}
(f^m)''(\xi)-\alpha f(\xi)+\beta\xi f'(\xi)+\xi^{\sigma}f^{p}(\xi)=0, \quad \xi\in[0,\infty).
\end{equation}
Let us notice that the condition $\sigma>2(1-p)/(m-1)$ insures that the self-similarity exponents $\alpha$ and $\beta$ as in \eqref{SSexp} are well-defined and positive, thus it is the lower bound for $\sigma$ that will lead to finite time blow-up of the solutions. This is in \emph{strong contrast} with, for example, the autonomous case $\sigma=0$ where solutions exist and remain bounded globally in time, as shown for example in \cite{dPV90, dPV91}. We perform in the sequel a deep study of the previous ODE. We thus define what we understand by a \emph{good profile} below (similar to \cite{IS19a, IS19b}).
\begin{definition}\label{def1}
We say that a solution $f$ to the differential equation \eqref{SSODE} is a \textbf{good profile} if it fulfills one of the following two properties related to its behavior at $\xi=0$:

\indent (P1) $f(0)=a>0$, $f'(0)=0$.

\indent (P2) $f(0)=0$, $(f^m)'(0)=0$.

\noindent We say that a profile $f$ has an \textbf{interface} at some point $\xi_0\in(0,\infty)$ if
$$
f(\xi_0)=0, \qquad (f^m)'(\xi_0)=0, \qquad f>0 \ {\rm on} \ (\xi_0-\delta,\xi_0), \ {\rm for \ some \ } \delta>0.
$$
\end{definition}
This definition agrees with the well-known notion of an interface for a solution. Indeed, a solution $u$ to Eq. \eqref{eq1} in the form \eqref{SSform} with a profile $f$ having an interface at $\xi_0\in(0,\infty)$, has a time-moving interface at $|x|=s(t)=(T-t)^{-\beta}\xi_0$ for any $t\in(0,T)$. This is why, also as in our previous works \cite{IS19a, IS19b, IS20}, we will be interested in the \emph{good profiles with interface} according to Definition \ref{def1}. The range $0<p<1$ will introduce two big novelties with respect to the previously studied cases. First of all, the analysis will differ according to the sign of the expression $m+p-2$. This is a feature of the range $p<1$ which has been noticed also in the non-weighted case \cite{dPV90, dPV91}, and it is strongly related to the existence of the interfaces: indeed, it is shown in \cite{dPV90} that when $m+p-2\geq0$, finite speed of propagation holds true, thus good profiles with interface are expected, while for $m+p-2<0$ the speed of propagation of the supports becomes infinite, thus the interfaces disappear and a solution (even if the initial condition $u_0$ is compactly supported) becomes positive immediately. The second important novelty in the range \eqref{range.exp} with respect to the results in our previous works is the existence of \textbf{two different interface behaviors}. Indeed, even a formal calculation on Eq. \eqref{SSODE} gives that a solution may develop an interface at a point $\xi_0>0$ in the following two forms:

$\bullet$ $f(\xi)\sim(\xi_0-\xi)^{1/(m-1)}$, as $\xi\to\xi_0$. This is the standard interface behavior inherited from the diffusion (the Barenblatt solutions to the standard porous medium equation have this type of contact at the interface) and will be called \emph{interface of Type I} in the text.

$\bullet$ $f(\xi)\sim(\xi_0-\xi)^{1/(1-p)}$, as $\xi\to\xi_0$. This is a new interface behavior that is interesting and will be analyzed in the paper. We will call it for convenience \emph{interface of Type II}.

We are now in a position to state, one by one, the main results of this paper. We begin with a general existence result of good profiles with both kinds of interfaces.
\begin{theorem}[Existence of good profiles with interface]\label{th.exist}
For any $m>1$, $p\in(0,1)$ such that $m+p>2$ and $\sigma>2(1-p)/(m-1)$, there exists at least one good profile with interface of Type I and one good profile with interface of Type II to Eq. \eqref{SSODE}, in the sense of the previous definitions.
\end{theorem}
This shows that the patterns for blow-up to Eq. \eqref{eq1} may be different. We will discuss about this further when we introduce the \emph{interface equation} and we show that the two types of interface are \emph{strongly different} with respect to the interface equation. For now, we continue with our main results, particularizing them with respect to their behavior both at the starting point $\xi=0$ and at their interface point (Type I or Type II). We first have a general result concerning profiles with interface of Type II.
\begin{theorem}[Good profiles with interface of Type II]\label{th.type2}
For any $m>1$, $p\in(0,1)$ such that $m+p>2$ and $\sigma>2(1-p)/(m-1)$, there exist good profiles with interface of Type II and satisfying property (P2) in Definition \ref{def1}. More precisely, these good profiles behave near the origin in the following way:
\begin{equation}\label{beh.02}
f(\xi)\sim K\xi^{(\sigma+2)/(m-p)}, \ \ K>0, \qquad {\rm as} \ \xi\to0,
\end{equation}
and the corresponding self-similar solutions blow up in finite time $t=T$ only at the space infinity.
\end{theorem}
This is an interesting result completing Theorem \ref{th.exist} for profiles with interface of Type II. Let us notice that in our previous papers \cite{IS19a, IS19b} we obtained a similar result for equations presenting either algebraic (power-like) or exponential spatial decay as $|x|\to\infty$, which were good self-similar solutions too but without interfaces. It seems that for $p>1$, the "ancient" spatial decay as $|x|\to\infty$ (which exists for $p\geq1$) converts into the behavior of interface of Type II, due to the change of sign of $1-p$. However, the co-existence of different interfaces is a very noticeable phenomenon in our case.

In the same line as in our previous works, a \emph{strong difference} of the graph and behavior of profiles holds true with respect to the magnitude of $\sigma$. Indeed, when $\sigma$ is sufficiently closer to its lower limit $2(1-p)/(m-1)$, we find that all the profiles $f(\xi)$ starting with $f(0)=0$ form interfaces of Type II. More precisely:
\begin{theorem}[Good profiles with interface for $\sigma$ small]\label{th.small}
For any $m>1$, $p\in(0,1)$ such that $m+p>2$, we have the following results:

(a) There exists $\sigma_0\in(2(1-p)/(m-1),\infty)$ such that for any $\sigma\in(2(1-p)/(m-1),\sigma_0)$, \emph{all} the good profiles satisfying property (P2) in Definition \ref{def1} present an interface of Type II. Moreover, for any $\sigma\in(2(1-p)/(m-1),\sigma_0)$ there exist good profiles with interface of Type I and satisfying property (P1) in Definition \ref{def1}. The corresponding self-similar solutions to the latter profiles blow up globally (that is, at any point $x\in\real$) in finite time $t=T$.

(b) There exists $\sigma_*\in(2(1-p)/(m-1),\infty)$ such that when $\sigma=\sigma_*$, there exists a unique profile $f_*(\xi)$ to Eq. \eqref{SSODE} such that
\begin{equation}\label{beh.01}
f_*(\xi)\sim\left[\frac{m-1}{2m(m+1)}\right]^{1/(m-1)}\xi^{2/(m-1)}, \qquad {\rm as} \ \xi\to0,
\end{equation}
and $f_*$ presents an interface of Type I at some $\xi_0\in(0,\infty)$. The corresponding self-similar solutions to this profile blow-up globally at time $t=T$.
\end{theorem}
For $\sigma$ sufficiently large things are different. There will be no longer good profiles with interface presenting the behavior in \eqref{beh.01} at $\xi=0$. But we can characterize more precisely the profiles with interface of Type I.
\begin{theorem}[Good profiles with interface for $\sigma$ large]\label{th.large}
For any $m>1$, $p\in(0,1)$ such that $m+p>2$, there exists $\sigma_1>2(1-p)/(m-1)$ sufficiently large such that for any $\sigma\in(\sigma_1,\infty)$, there exist blow-up profiles satisfying property (P2) in Definition \ref{def1}, presenting the behavior \eqref{beh.02} as $\xi\to0$ and having an interface of Type I at some point $\xi=\xi_0\in(0,\infty)$. The corresponding self-similar solutions to these profiles blow up in finite time $t=T$ only at space infinity.
\end{theorem}

\noindent \textbf{Remark.} A first interesting point to be emphasized is that \emph{finite time blow-up occurs} also for $m>1$ and $p<1$. This is due to the strong influence of the weight $|x|^{\sigma}$, since it is not true in the non-weighted case (that is, when $\sigma=0$). Moreover, the blow-up set strongly differs with $\sigma$: for $\sigma$ relatively small (close to the lower limit $2(1-p)/(m-1)$), both self-similar solutions presenting a global blow-up and self-similar solutions presenting blow-up only at the space infinity (while they remain bounded at any $|x|$ fixed) do exist, while for $\sigma$ very large it is likely that \emph{all} self-similar solutions blow up at the space infinity. Such a difference with respect to the blow-up behavior was deduced also in our previous papers \cite{IS19a, IS19b} dealing with exponents $p=1$, respectively $1<p<m$, where we rigorously define the blow-up set and describe in greater detail the blow-up at the space infinity.

Finally, let us notice that all the previous theorems hold true in the hypothesis that $m+p>2$. Due to significant differences in the techniques of the proofs and in some of the results, we separate the very critical case $m+p=2$ to a companion paper \cite{IS20crit}. However, the subcritical case $m+p<2$ is very simple and striking:
\begin{theorem}[Non-existence for $m+p<2$]\label{th.non}
Let $m>1$, $p\in(0,1)$ such that $m+p<2$ and let $\sigma>2(1-p)/(m-1)$. Then there exist \emph{no good blow-up profiles} with or without interface.
\end{theorem}
The fact that blow-up profiles with interface do no longer exist in the case $m+p<2$ was expected, since it is known \cite{dPV90, dPV91} that solutions propagate with infinite speed in this case. But the general non-existence result is very striking, as there is no different behavior at all (such as a tail as $|x|\to\infty$ for example) to replace the interface behavior in this case. The result in Theorem \ref{th.non} is strongly related to a more general non-existence result for a similar equation that will be given in \cite{IMS20}, but we nevertheless give a full proof of the non-existence for self-similar profiles using the techniques of this paper.

\medskip

\noindent \textbf{The interface equation. Differences between interfaces of Type I and Type II}. We discuss here, at a formal level, the interface equation satisfied by the interfaces of Type I and of Type II when $m>1$, $p\in(0,1)$ and $m+p>2$ to show the difference of behavior of the two interfaces. Let us recall that for a compactly supported and radially symmetric solution $u$ to Eq. \eqref{eq1}, the \emph{interface (or free boundary)} of $u$ is defined as the supremum of its support at time $t>0$
$$
s(t)=\sup\{|x|: x\in\real, u(x,t)>0\}, \qquad t>0.
$$
For solutions presenting an interface of Type I, we pass as usual to the equation for the \emph{pressure variable}
$$
v(x,t)=\frac{m}{m-1}u^{m-1}(x,t)
$$
and obtain the equation satisfied by $v$ (similar to \cite{dPV90, dPV91})
\begin{equation}\label{pressure}
v_t=(m-1)vv_{xx}+v_{x}^2+m\left(\frac{m-1}{m}\right)^{(m+p-2)/(m-1)}|x|^{\sigma}v^{(m+p-2)/(m-1)}.
\end{equation}
Starting from the obvious equality $v(s(t),t)=0$ and formally differentiating with respect to $t$ we readily get that
$$
s'(t)=-\frac{v_t(s(t),t)}{v_x(s(t),t)}.
$$
Replacing now $v_t$ by the right-hand side in Eq. \eqref{pressure} and working on self-similar profiles it is easy to check that the terms in $v_{xx}$ and the last one vanish at $s(t)$ (since $m+p-2>0$) and we remain with the \emph{standard interface equation}
\begin{equation}\label{interfeq1}
s'(t)=-v_{x}(s(t),t),
\end{equation}
similar to the one fulfilled for example by the solutions to the standard porous medium equation, thus the reaction term involving $\sigma$ plays no role here. On the other hand, for an interface of Type II, we follow an idea used for traveling wave solutions stemming from Herrero and V\'azquez \cite{HV88} (see also \cite{dPS00, dPS02}) and introduce the following change of function specific to the range $p<1$
$$
w(x,t)=\frac{1}{1-p}u(x,t)^{1-p}.
$$
The equation solved by $w$ is
\begin{equation}\label{pres2}
\begin{split}
w_t&=\frac{m}{1-p}(1-p)^{(m-p)/(1-p)}w^{(m-1)/(1-p)}w_{xx}\\
&+\frac{m(m+p-1)}{(1-p)^2}(1-p)^{(m-p)/(1-p)}w^{(m+p-2)/(1-p)}(w_x)^2+x^{\sigma},
\end{split}
\end{equation}
thus, by finding again that
$$
s'(t)=-\frac{w_t(s(t),t)}{w_x(s(t),t)}
$$
and replacing $w_t$ with the right-hand side of \eqref{pres2}, we readily find that on the self-similar solutions the first two terms cancel at the interface point $(s(t),t)$ and we are left with the (free) last term. Thus the interface equation for interfaces of Type II is
\begin{equation}\label{interfeq2}
s'(t)=-\frac{s(t)^{\sigma}}{w_{x}(s(t),t)},
\end{equation}
which reminds of the one obtained for the interface of the self-similar solutions to the reaction-convection-diffusion equation in \cite{dPS02} but in our case it also \emph{strongly depends on $\sigma$}. We notice that equations \eqref{interfeq1} and \eqref{interfeq2} are very different, which shows that the Type II interface behaviors is novel and qualitatively interesting, while the Type I behavior inherits the properties from the one with $\sigma=0$.

\section{The phase space when $m+p>2$. Proof of Theorem \ref{th.type2}}\label{sec.type2}

We will consider from now on, unless if the contrary is specified, that $m+p>2$. The main tool in the proofs of the main results of the present work is a thorough analysis of a phase space associated to an autonomous dynamical system which is equivalent to the non-autonomous equation of profiles \eqref{SSODE}. Thus we transform Eq. \eqref{SSODE} into an autonomous quadratic dynamical system by letting
\begin{equation}\label{PSvar1}
X(\eta)=\frac{m}{\alpha}\xi^{-2}f^{m-1}(\xi), \ Y(\eta)=\frac{m}{\alpha}\xi^{-1}f^{m-2}(\xi)f'(\xi), \ Z(\eta)=\frac{m}{\alpha^2}\xi^{\sigma-2}f^{m+p-2}(\xi),
\end{equation}
where we recall that $\alpha$ (and also $\beta$) is defined in \eqref{SSexp} and the new independent variable $\eta=\eta(\xi)$ is defined through the differential equation
$$
\frac{d\eta}{d\xi}=\frac{\alpha}{m}\xi f^{1-m}(\xi).
$$
The differential equation \eqref{SSODE} transforms into the system
\begin{equation}\label{PSsyst1}
\left\{\begin{array}{ll}\dot{X}=X[(m-1)Y-2X],\\
\dot{Y}=-Y^2-\frac{\beta}{\alpha}Y+X-XY-Z,\\
\dot{Z}=Z[(m+p-2)Y+(\sigma-2)X].\end{array}\right.
\end{equation}
Let us remark at this point that the system \eqref{PSsyst1} differs from the phase space system analyzed in both \cite{IS19a, IS19b}. Indeed, the variable $Z(\eta)=\xi^{\sigma}f^{p-1}(\xi)$ used in the above quoted papers is no longer useful for $p<1$ as it sends the interface behaviors to infinity. We have to work instead with the new system \eqref{PSsyst1} which is very well adapted for the case $p<1$ and $m+p>2$. However, it has a further technical difficulty stemming from the fact that sometimes the coefficient $\sigma-2$ in the third equation might be negative. Notice for now that the planes $\{X=0\}$ and $\{Z=0\}$ are invariant for the system and that $X\geq0$, $Z\geq0$, only $Y$ being allowed to change sign. We easily find that for $m+p>2$ there are three critical points in the finite plane:
$$
P_0=(0,0,0), \ P_1=\left(0,-\frac{\beta}{\alpha},0\right), \ {\rm and} \ P_2=\left(\frac{m-1}{2(m+1)\alpha},\frac{1}{(m+1)\alpha},0\right).
$$
We devote the present section to the local analysis near the strongly non-hyperbolic point $P_0$, which is rather technical and quite complex. The local analysis near the points $P_1$, $P_2$ and of the critical points at infinity is left for the next section.

\medskip

\noindent \textbf{Local analysis of the point $P_0$.} The linearization of the system \eqref{PSsyst1} near $P_0$ has the matrix
$$
M(P_0)=\left(
      \begin{array}{ccc}
        0 & 0 & 0 \\
        1 & -\frac{\beta}{\alpha} & -1 \\
        0 & 0 & 0 \\
      \end{array}
    \right)
$$
thus it has a one-dimensional stable manifold and a two-dimensional center manifold. In order to study the center manifold and the flow on it, we perform the change of variable
$$
T:=\frac{\beta}{\alpha}Y-X+Z
$$
and after straightforward (but rather tedious) calculations, the system in variables $(X,T,Z)$ becomes
\begin{equation}\label{interm1}
\left\{\begin{array}{ll}\dot{X}&=\frac{1}{\beta}X[X+(m-1)\alpha T-(m-1)\alpha Z],\\
\dot{T}&=-\frac{\beta}{\alpha}T-\frac{\alpha}{\beta}T^2-\frac{\alpha(m+1)+\beta}{\beta}XT-\frac{\alpha(m+p)}{\beta}TZ\\
&-\frac{m\alpha-\beta}{\beta}X^2+\frac{3\beta+2\alpha+3}{\beta}XZ-\frac{\alpha(m+p-1)}{\beta}Z^2,\\
\dot{Z}&=\frac{1}{\beta}Z[2X+(m+p-2)\alpha T-(m+p-2)\alpha Z].\end{array}\right.
\end{equation}
We can now apply the local center manifold theorem \cite[Theorem 1, Section 2.12]{Pe} to find (rather easily, by taking off the third order terms in the equation of $T$ in the system \eqref{interm1}) that the center manifold has the equation
$$
T(X,Z)=\frac{\alpha}{\beta}\left[-\frac{m\alpha-\beta}{\beta}X^2+\frac{3\beta+2\alpha+3}{\beta}XZ-\frac{\alpha(m+p-1)}{\beta}Z^2\right]+O(|(X,Z)|^3)
$$
and the flow on the center manifold is given by the almost homogeneous quadratic system
\begin{equation}\label{interm2}
\left\{\begin{array}{ll}\dot{X}&=\frac{1}{\beta}X[X-(m-1)\alpha Z]+O(|(X,Z)|^3),\\
\dot{Z}&=\frac{1}{\beta}Z[2X-(m+p-2)\alpha Z]+O(|(X,Z)|^3),\end{array}\right.
\end{equation}
which is easily obtained by keeping only the quadratic terms in the equations fulfilled by $X$, $Z$ in the system \eqref{interm1}. In order to study the flow given by this system in a neighborhood of the point $(X,Z)=(0,0)$ we need to use the rather complicated but complete classification of the (2,2)-homogeneous dynamical systems established in the renowned paper by Date \cite{Date79}. This study will lead us directly to the proof of Theorem \ref{th.type2}.

\medskip

\noindent \begin{proof}[Proof of Theorem \ref{th.type2}] As explained above, the proof consists esentially in the study of the homogeneous quadratic part of the system \eqref{interm2} above, using the theory from \cite{Date79}. Putting aside the common factor $1/\beta$ which can be absorbed by a change in the indepedent variable, we have to study the quadratic system
\begin{equation}\label{interm3}
\left\{\begin{array}{ll}\dot{X_1}&=X_1^2-(m-1)\alpha X_1X_2,\\
\dot{X_2}&=2X_1X_2-(m+p-2)\alpha X_2^2,\end{array}\right.
\end{equation}
We will use from now on the same notation as in \cite{Date79}, where the reader can find more details on the classification that follows below. The idea is to compute several important invariants associated to the system \eqref{interm2} and study their sign in order to get the phase portrait near the origin. The tensor of the coefficients of the system \eqref{interm2} (written in homogeneous form) is given by
\begin{equation*}
\begin{split}
&P_{11}^1=1, \quad P_{12}^1=P_{21}^1=-\frac{(m-1)\alpha}{2}, \quad P_{22}^1=0,\\
&P_{11}^2=0, \quad P_{12}^2=P_{21}^2=1, \quad P_{22}^2=-(m+p-2)\alpha,
\end{split}
\end{equation*}
and according to \cite{Date79} we decompose this tensor $P^{k}_{\lambda,\mu}$ into its vector part $p_{\lambda}$ and its tensor part $Q^{k}_{\lambda,\mu}$ using the formulae
\begin{equation}\label{interm4}
p_{\lambda}=\sum_{k=1}^{2}P_{\lambda,k}^{k}, \quad Q^{k}_{\lambda,\mu}=P^{k}_{\lambda,\mu}-\frac{1}{3}(\delta_{\lambda,k}p_{\mu}-\delta_{\mu,k}p_{\lambda}),
\end{equation}
where $\delta_{i,j}=1$ if $i=j$ and zero otherwise. In our case, it is easy to check that the vector part contains
$$
p_1=P_{11}^1+P_{12}^2=2, \quad p_2=P_{21}^1+P_{22}^2=-\frac{\alpha}{2}(3m+2p-5)
$$
and the tensor part is given by
\begin{equation*}
\begin{split}
&Q_{11}^1=P_{11}^1-\frac{2}{3}p_1=1-\frac{4}{3}=-\frac{1}{3},\\
&Q_{12}^1=Q_{21}^1=P_{12}^1-\frac{1}{3}p_2=-\frac{\alpha}{3}(1-p),\\
&Q_{22}^1=P_{22}^1=0, \ Q_{11}^2=P_{11}^2=0, \ Q_{12}^2=Q_{21}^2=P_{12}^2-\frac{1}{3}p_1=\frac{1}{3},\\
&Q_{22}^2=P_{22}^2-\frac{2}{3}p_2=\frac{\alpha}{3}(1-p).
\end{split}
\end{equation*}
With the help of these values, we further compute the Hessian of the fundamental cubic form associated to the system with the general formulae
$$
h^{k,l}=\frac{1}{2}\sum_{\mu,\nu,\rho,\sigma=1}^2\epsilon^{\mu,\nu}\epsilon^{\rho,\sigma}Q_{\mu,\rho}^kQ_{\nu,\sigma}^{l},
$$
where $\epsilon^{11}=\epsilon^{22}=0$ and $\epsilon^{12}=-\epsilon^{21}=-1$. After some easy calculations, we obtain that in our case the Hessian writes
\begin{equation*}
h^{11}=-\frac{\alpha^2}{9}(1-p)^2, \quad h^{12}=h^{21}=\frac{\alpha}{18}(1-p) \quad h^{22}=-\frac{1}{9}.
\end{equation*}
With all these numbers, we are now ready to compute the fundamental scalar invariants of degree 2 called $D$, $H$ and $F$ in \cite{Date79}, which are the basis of the classification of all the phase portraits. We thus have
$$
H=h^{11}p_1^2+2h^{12}p_1p_2+h^{22}p_2^2=-\frac{\alpha^2}{9}\left[3(p-1)^2+\frac{9}{4}(m-1)^2\right],
$$
$$
D=-2\sum_{k,l,\mu,\nu=1}^2\epsilon_{k,l}\epsilon_{\mu,\nu}h^{k,\mu}h^{l,\nu}=-\frac{1}{27}\alpha^2(1-p)^2,
$$
where in the last formula we denoted $\epsilon_{11}=\epsilon_{22}=0$, $\epsilon_{12}=-\epsilon_{21}=-1$, and
$$
F=\sum_{\mu,k,\rho,l=1}^2\epsilon^{\mu,k}\epsilon^{\rho,l}\left(\sum_{\sigma=1}^2Q_{kl}^{\sigma}p_{\mu}p_{\rho}p_{\sigma}\right)=-\frac{\alpha^2}{2}(3m+2p-5)(3m-2p-1).
$$
With these scalar invariants we can finally introduce the general set of invariants introduced by Date and Iri in \cite{DI76} having the general expression
$$
K_m=F+9(-2)^{m-3}H-27(-8)^{m-3}D, \quad m=1,2,...
$$
In particular, it is easy to calculate $K_2$ and $K_3$, more precisely
$$
K_2=F-\frac{9}{2}H+\frac{27}{8}D=-\frac{27}{8}\alpha^2(m+p-2)(m-p)<0,
$$
and
$$
K_3=F+9H-27D=-\frac{27}{4}\alpha^2(m-1)^2<0.
$$
According to the general classification in \cite[p. 327]{Date79} we find that we are in the case $D<0$, $K_2<0$ and $K_3<0$, which corresponds to the phase portrait no. 8 in \cite[Figure 8, p. 329]{Date79}. We thus infer that the local behavior near the origin in the system \eqref{interm2} presents an elliptic sector in a sufficiently small neighborhood of the origin, as shown in Figure \ref{fig1}, hence there exist orbits in the phase space which go out and then enter the critical point $P_0$ along the center manifold.

\begin{figure}[ht!]
  \begin{center}
  \includegraphics[width=10cm,height=7cm]{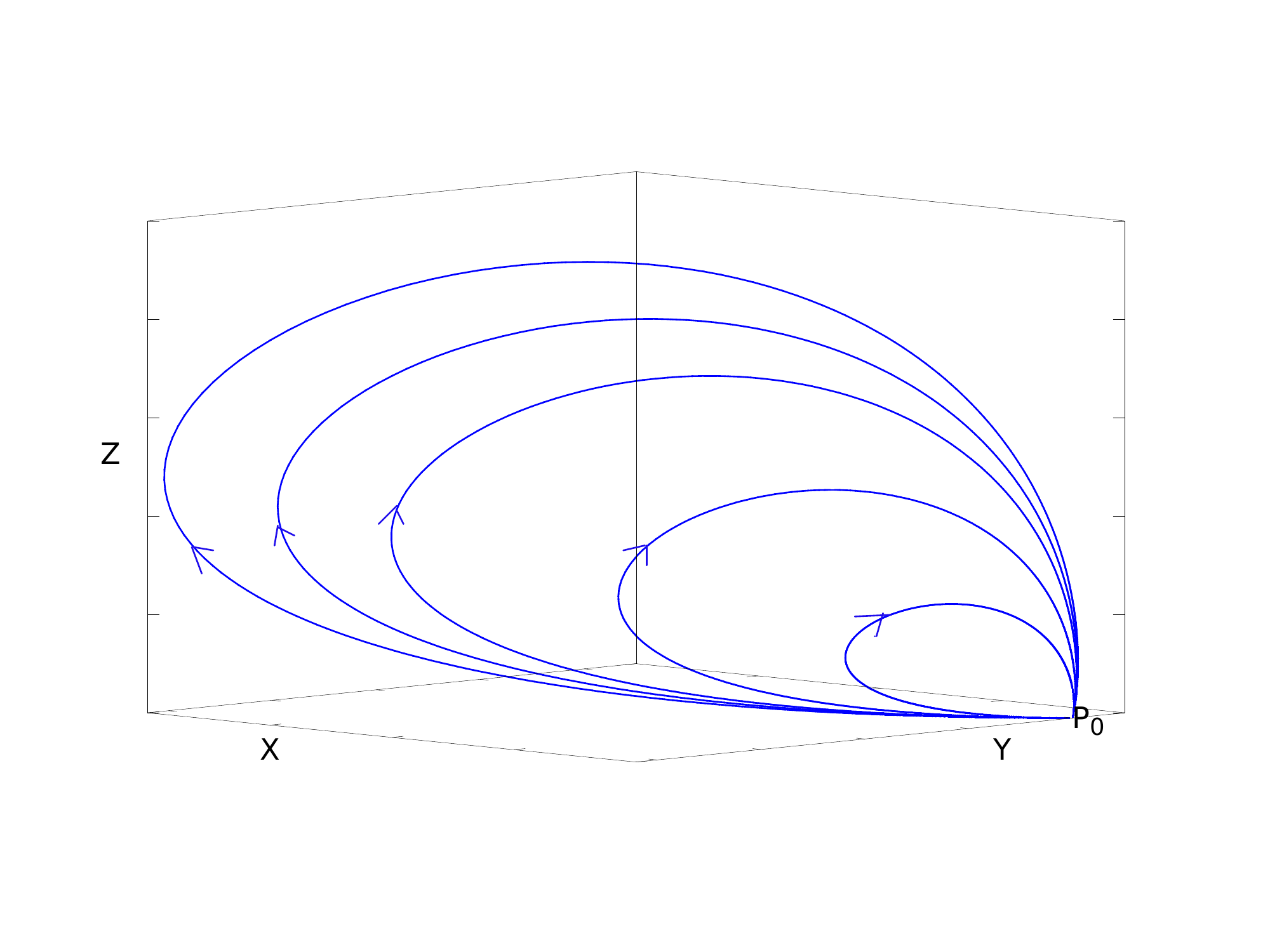}
  \end{center}
  \caption{Local behavior of the system \eqref{PSsyst1} with the elliptic sector near the origin. Numerical experiment for $m=3$, $p=0.5$ and $\sigma=1$}\label{fig1}
\end{figure}

Coming back to variables $(X,T,Z)$, the profiles contained in these connection have $T\sim0$ or equivalently, undoing the change of variable,
$$
\frac{\beta}{\alpha}Y-X+Z\sim0
$$
which in terms of profiles becomes
\begin{equation}\label{interm5}
f'(\xi)-\frac{\alpha}{\beta}\xi^{-1}f(\xi)+\frac{1}{\beta}\xi^{\sigma-1}f^{p}(\xi)\sim0.
\end{equation}
We discard at this point the possibility that the limit behavior in \eqref{interm5} is taken as $\xi\to\infty$. Indeed, assume for contradiction that \eqref{interm5} holds true as $\xi\to\infty$. Since $X(\xi)\to0$, it follows that $f(\xi)<K\xi^{2/(m-1)}$ for any $K>0$ (for $\xi$ large enough), whence
$$
\frac{\xi^{\sigma-1}f(\xi)^p}{\xi^{-1}f(\xi)}=\xi^{\sigma}f^{p-1}(\xi)>K\xi^{\frac{\sigma(m-1)+2(p-1)}{m-1}}\to\infty, \ {\rm as} \ \xi\to\infty.
$$
On the other hand, we can just integrate the quadratic part of the system \eqref{interm2}. More precisely, letting
$$
\frac{dZ}{dX}=\frac{Z(2X-(m+p-2)\alpha Z)}{X(X-(m-1)\alpha Z)},
$$
we obtain a homogeneous differential equation that can be explicitly integrated to find the general solution
$$
X=\frac{KW}{(1+(1-p)\alpha W)^{(m-p)/(1-p)}}, \quad W=\frac{Z}{X}, \ K\in\real,
$$
or equivalently in terms of profiles
$$
\xi^{-2}f^{m-1}(\xi)\sim\frac{K\xi^{\sigma}f^{p-1}(\xi)}{(1+(1-p)\alpha\xi^{\sigma}f(\xi)^{p-1})^{(m-p)/(1-p)}}, \quad {\rm as} \ \xi\to\infty
$$
which can be written also as
\begin{equation}\label{interm6}
(1+(1-p)\alpha\xi^{\sigma}f(\xi)^{p-1})^{(m-p)/(1-p)}\sim K\xi^{\sigma+2}f(\xi)^{p-m}, \quad {\rm as} \ \xi\to\infty.
\end{equation}
But we just noticed above that $\xi^{\sigma}f(\xi)^{p-1}\to\infty$ as $\xi\to\infty$, and since $(m-p)/(1-p)>1$, by keeping the dominating orders in $\xi$ in \eqref{interm6} we get
$$
C\xi^{\sigma(m-p)/(1-p)}f(\xi)^{p-m}\sim K\xi^{\sigma+2}f(\xi)^{p-m}, \quad {\rm as} \ \xi\to\infty.
$$
But the latter is equivalent after simplifications to
$$
\xi^{[\sigma(m-1)+2(1-p)]/(1-p)}\sim\frac{K}{C}, \quad {\rm as} \ \xi\to\infty,
$$
which is a contradiction since $(\sigma(m-1)+2(p-1))/(1-p)>0$. Thus, we cannot enter $P_0$ as $\xi\to\infty$ in terms of profiles. The remaining possibilities are that \eqref{interm5} holds true as either $\xi\to0$ or as $\xi\to\xi_0\in(0,\infty)$. The general solution to the equation \eqref{interm5} with exact equality to zero is
$$
f_0(\xi)=\xi^{(\sigma+2)/(m-p)}\left[C-(1-p)\xi^{\sigma}\right]^{1/(1-p)}.
$$
On the one hand the differential equation with solution $f_0$ is an approximation for the true equation \eqref{interm5}. On the other hand the orbits cannot go out from a critical point as $\xi\to\xi_0=(C/(1-p))^{1/\sigma}$. Assume for contradiction that this were the case. It then follows that $f'(\xi_0)>0$ and
$$
\frac{1}{\beta}\xi^{\sigma-1}f(\xi)^p-\frac{\alpha}{\beta}\xi^{-1}f(\xi)=\frac{1}{\beta}\xi^{-1}f(\xi)\left(\xi^{\sigma}f(\xi)^{p-1}-\alpha\right)>0
$$
in a right-neighborhood of $\xi_0$, since $p-1<0$ and $f(\xi_0)=0$. We thus get a contradiction to \eqref{interm5}. It thus follows that the profiles contained in the orbits going out of $P_0$ behave as in \eqref{beh.02} as $\xi\to0$. By similar arguments and according to the form of $f_0$, the profiles contained in orbits entering $P_0$ have an interface of Type II at some point $\xi_0=(C/(1-p))^{1/\sigma}\in(0,\infty)$. Since all these hold true for any $\sigma>2(1-p)/(m-1)$, any profile contained in such an elliptic orbit satisfies Theorem \ref{th.type2}.
\end{proof}

\section{Local analysis of the critical points}\label{sec.local}

In this section we complete the local analysis of the critical points of the system \eqref{PSsyst1} both in the finite space and at infinity. This part is rather similar to \cite[Section 2]{IS19b}. Let us recall that the critical point $P_0$ was studied in Section \ref{sec.type2}. We start with the remaining finite critical points $P_1$ and $P_2$.
\begin{lemma}[Local analysis of the point $P_1$]\label{lem.P1}
The system \eqref{PSsyst1} in a neighborhood of the critical point $P_1$ has a one-dimensional unstable manifold and a two-dimensional stable manifold. The orbits entering $P_1$ on the stable manifold contain profiles such that
\begin{equation}\label{Beh.P1}
f(\xi)\sim\left[K-\frac{\beta(m-1)}{2m}\xi^2\right]^{1/(m-1)}, \quad K>0,
\end{equation}
for $\xi\to\xi_0=\sqrt{2mK/(m-1)\beta}\in(0,\infty)$. Thus, this point gathers the Type I interface behavior.
\end{lemma}
\begin{proof}
The linearization of the system \eqref{PSsyst1} near this critical point has the matrix:
$$M(P_1)=\left(
  \begin{array}{ccc}
    -\frac{\beta(m-1)}{\alpha} & 0 & 0 \\
    1+\frac{\beta}{\alpha} & \frac{\beta}{\alpha} & -1 \\
    0 & 0 & -\frac{(m+p-2)\beta}{\alpha} \\
  \end{array}
\right)
$$
with eigenvalues
$$
\lambda_1=-\frac{\beta(m-1)}{\alpha}<0, \qquad \lambda_2=\frac{\beta}{\alpha}>0, \qquad \lambda_3=-\frac{(m+p-2)\beta}{\alpha}<0
$$
and respective eigenvectors (not normalized)
$$
e_1=\left(\frac{m\beta}{\alpha},-\left(1+\frac{\beta}{\alpha}\right),0\right), \ e_2=(0,1,0), \ e_3=\left(0,1,\frac{(m+p-1)\beta}{\alpha}\right).
$$
We then have a two-dimensional stable manifold with orbits entering the point $P_1$ in the phase space and (as it is easy to check) a unique orbit going out of $P_1$ along the $Y$-axis. We look for the profiles contained in the orbits entering $P_1$ on the two-dimensional stable manifold. We infer from the change of variables \eqref{PSvar1} that
\begin{equation}\label{interm7}
Y(\xi)=\frac{m\xi^{-1}}{\alpha(m-1)}(f^{m-1})'(\xi)\sim-\frac{\beta}{\alpha}
\end{equation}
on the orbits entering $P_1$. We show first that \eqref{interm7} holds true as $\xi\to\xi_0\in(0,\infty)$. Indeed, assume first for contradiction that \eqref{interm7} holds true as $\xi\to\infty$. Then
$$
\lim\limits_{\xi\to\infty}\frac{(f^{m-1})'(\xi)}{\xi}=-\frac{\beta(m-1)}{m}
$$
whence by L'Hospital rule we deduce that
$$
\lim\limits_{\xi\to\infty}X(\xi)=\frac{m}{\alpha}\lim\limits_{\xi\to\infty}\frac{f^{m-1}(\xi)}{\xi^2}=-\frac{\beta(m-1)}{2\alpha}.
$$
But this is a contradiction with the fact that $X(\xi)\to0$ on an orbit entering $P_1$. Assume now for contradiction that \eqref{interm7} holds true as $\xi\to0$. A similar argument based on the L'Hospital rule leads to a similar contradiction as before, after noticing that $X(\xi)\to0$ implies $f^{m-1}(\xi)\to0$ as $\xi\to0$ and thus the L'Hospital rule can be applied for the function $f^{m-1}(\xi)/\xi^2$ giving $X(\xi)$ modulo a constant. We thus conclude that \eqref{interm7} holds true as $\xi\to\xi_0$ for some $\xi_0\in(0,\infty)$ and readily get the behavior described in \eqref{Beh.P1} by direct integration.
\end{proof}
We complete the analysis of the critical points in the plane by performing the local analysis near $P_2$.
\begin{lemma}[Local analysis of the point $P_2$]\label{lem.P2}
The system \eqref{PSsyst1} in a neighborhood of the critical point $P_2$ has a two-dimensional stable manifold and a one-dimensional unstable manifold. The stable manifold is contained in the invariant plane $\{Z=0\}$. There exists a unique orbit going out of $P_2$, containing profiles such that
\begin{equation}\label{Beh.P2}
f(0)=0, \quad f(\xi)\sim\left[\frac{m-1}{2m(m+1)}\right]^{1/(m-1)}\xi^{2/(m-1)}, \ \ {\rm as} \ \xi\to0.
\end{equation}
\end{lemma}
\begin{proof}
The linearization of the system \eqref{PSsyst1} near the critical point $P_2$ has the matrix
$$
M(P_2)=\frac{1}{2(m+1)\alpha}\left(
  \begin{array}{ccc}
    -2(m-1) & (m-1)^2 & 0 \\
    2(m+1)\alpha-2 & -2\beta(m+1)-(m+3) & -2(m+1)\alpha \\
    0 & 0 & \sigma(m-1)+2(p-1) \\
  \end{array}
\right)
$$
with eigenvalues $\lambda_1$, $\lambda_2$ and $\lambda_3$ such that
$$
\lambda_1+\lambda_2=-\frac{2(m-1)+2(m+1)\beta}{2(m+1)\alpha}<0, \ \lambda_1\lambda_2=\frac{m-1}{2(m+1)\alpha^2}>0
$$
whence $\lambda_1$, $\lambda_2<0$ and
$$
\lambda_3=\frac{\sigma(m-1)+2(p-1)}{2(m+1)\alpha}>0.
$$
It is easy to check (by computing the eigenvectors corresponding to $\lambda_1$ and $\lambda_2$ and noticing that both have the $Z$-component zero) that the two-dimensional stable manifold is contained in the invariant plane $\{Z=0\}$. Similarly as in \cite[Lemma 2.3]{IS19a} we conclude that there exists an unique orbit going out of $P_2$ towards the interior of the phase space, tangent to the eigenvector corresponding to $\lambda_3$ which is
$$
e_3=\left(\frac{2(m-1)^2(m+1)\alpha}{D},\frac{2(m+1)[(m-1)(\sigma+2)+2(p-1)]}{D},1\right),
$$
where
$$
D=-(m-1)^2\sigma^2-(m-1)(3m+4p-3)\sigma-4m^2-4mp-4p^2+4m+8p<0.
$$
The local behavior \eqref{Beh.P2} of the profiles contained in the orbit going out of $P_2$ is obtained from the fact that
\begin{equation}\label{interm8}
X(\xi)=\frac{m}{\alpha}\frac{f^{m-1}(\xi)}{\xi^2}\sim\frac{m-1}{2(m+1)\alpha},
\end{equation}
which is obvious that cannot hold true as $\xi\to\xi_0\in(0,\infty)$ (in such a case $f(\xi)$ would start from a positive constant) and again a contradiction based on the L'Hospital rule similar to the ones in the proof of Lemma \ref{lem.P1} discards the possibility that $\xi\to\infty$. Thus \eqref{interm8} holds true necessarily as $\xi\to0$ and this is equivalent to the claimed local behavior \eqref{Beh.P2}.
\end{proof}

\medskip

\noindent \textbf{Local analysis of the critical points at infinity}. Together with the finite critical points already analyzed, in order to understand the global picture of the phase space associated to the system \eqref{PSsyst1}, we need to analyze its critical points at the space infinity. To this end, we pass to the Poincar\'e hypersphere according to the theory in \cite[Section 3.10]{Pe}. We thus introduce the new variables $(\overline{X},\overline{Y},\overline{Z},W)$ by
$$
X=\frac{\overline{X}}{W}, \ Y=\frac{\overline{Y}}{W}, \ Z=\frac{\overline{Z}}{W}
$$
and we derive from \cite[Theorem 4, Section 3.10]{Pe} that the critical points at space infinity lie on the equator of the Poincar\'e hypersphere, hence at points
$(\overline{X},\overline{Y},\overline{Z},0)$ where $\overline{X}^2+\overline{Y}^2+\overline{Z}^2=1$ and the following system is fulfilled:
\begin{equation}\label{Poincare1}
\left\{\begin{array}{ll}\overline{X}Q_2(\overline{X},\overline{Y},\overline{Z})-\overline{Y}P_2(\overline{X},\overline{Y},\overline{Z})=0,\\
\overline{X}R_2(\overline{X},\overline{Y},\overline{Z})-\overline{Z}P_2(\overline{X},\overline{Y},\overline{Z})=0,\\
\overline{Y}R_2(\overline{X},\overline{Y},\overline{Z})-\overline{Z}Q_2(\overline{X},\overline{Y},\overline{Z})=0,\end{array}\right.
\end{equation}
where $P_2$, $Q_2$ and $R_2$ are the homogeneous second degree parts of the terms in the right hand side of the system \eqref{PSsyst1}, that is
\begin{equation*}
\begin{split}
&P_2(\overline{X},\overline{Y},\overline{Z})=\overline{X}[(m-1)\overline{Y}-2\overline{X}],\\
&Q_2(\overline{X},\overline{Y},\overline{Z})=-\overline{Y}^2-\overline{X}\overline{Y},\\
&R_2(\overline{X},\overline{Y},\overline{Z})=\overline{Z}[(m+p-2)\overline{Y}+(\sigma-2)\overline{X}].
\end{split}
\end{equation*}
We thus find that the system \eqref{Poincare1} becomes
\begin{equation}\label{Poincare2}
\left\{\begin{array}{ll}\overline{X}\overline{Y}(\overline{X}-m\overline{Y})=0,\\
\overline{X}\overline{Z}[\sigma\overline{X}-(1-p)\overline{Y}]=0,\\
\overline{Z}\overline{Y}[(m+p-1)\overline{Y}+(\sigma-1)\overline{X}]=0,\end{array}\right.
\end{equation}
Taking into account that we are considering only points with coordinates $\overline{X}\geq0$ and $\overline{Z}\geq0$, we find the following critical points at infinity (on the Poincar\'e hypersphere):
$$
Q_1=(1,0,0,0), \ \ Q_{2,3}=(0,\pm1,0,0), \ \ Q_4=(0,0,1,0), \ \
Q_5=\left(\frac{m}{\sqrt{1+m^2}},\frac{1}{\sqrt{1+m^2}},0,0\right)
$$
which are the same ones as for the phase space systems in \cite{IS19a, IS19b}. We perform next the local analysis near each one of them. This analysis follows closely the one in \cite{IS19a}, thus we will sometimes skip some details.
\begin{lemma}[Local analysis of the point $Q_1$]\label{lem.Q1}
The critical point $Q_1=(1,0,0,0)$ in the Poincar\'e sphere is an unstable node. The orbits going out of this point into the finite part of the phase space contain profiles $f$ such that $f(0)=a>0$ and any possible value of $f'(0)$.
\end{lemma}
\begin{proof}
We apply part (a) of \cite[Theorem 5, Section 3.10]{Pe} to infer that the flow in a neighborhood of $Q_1$ is topologically equivalent to the flow in a neighborhood of the origin $(y,z,w)=(0,0,0)$ for the system
\begin{equation}\label{systinf1}
\left\{\begin{array}{ll}-\dot{y}=-y-w+my^2+\frac{\beta}{\alpha}yw+zw,\\
-\dot{z}=-\sigma z+(1-p)yz,\\
-\dot{w}=-2w+(m-1)yw,\end{array}\right.
\end{equation}
where the minus sign has been chosen in the system \eqref{systinf1} in order to match the direction of the flow. We deduce it from the first equation of the original system \eqref{PSsyst1},
$$
\dot{X}=X[(m-1)Y-2X],
$$
which gives $\dot{X}<0$ in a neighborhood of $Q_1$, taking into account that $|X/Y|\to+\infty$ near this point. Thus $Q_1$ is an unstable node since the linearization of the system \eqref{systinf1} near the origin has eigenvalues 1, 2 and $\sigma$. The local behavior of the profiles contained in the orbits going out of $Q_1$ is given by
\begin{equation*}
\frac{dz}{dw}\sim\frac{\sigma}{2}\frac{z}{w},
\end{equation*}
whence by integration $z\sim Cw^{\sigma/2}$. Coming back to the original variables and recalling that the projection of the Poincar\'e hypersphere has been done by dividing by the $X$ variable, we infer that $Z/X\sim CX^{-\sigma/2}$, $C>0$, thus
$$
\frac{m}{\alpha^2}\xi^{\sigma-2}f(\xi)^{m+p-2}\sim C\left(\frac{m}{\alpha}\right)^{(2-\sigma)/2}\xi^{\sigma-2}f(\xi)^{(m-1)(2-\sigma)/2}
$$
which leads easily to $f(\xi)\sim a$ for some $a>0$. Moreover, the latter holds true as $\xi\to0$, since at $Q_1$ we have $X\to\infty$. We thus get $f(0)=a>0$ with no further condition on the derivative $f'(0)$.
\end{proof}
\begin{lemma}[Local analysis of the points $Q_2$ and $Q_3$]\label{lem.Q23}
The critical points $Q_{2,3}=(0,\pm1,0,0)$ in the Poincar\'e hypersphere are an unstable node, respectively a stable node. The orbits going out of $Q_2$ to the finite part of the phase space contain profiles $f(\xi)$ such that there exists $\xi_0\in(0,\infty)$ with $f(\xi_0)=0$, $(f^m)'(\xi_0)>0$. The orbits entering the point $Q_3$ and coming from the finite part of the phase space contain profiles $f(\xi)$ such that there exists $\xi_0\in(0,\infty)$ with $f(\xi_0)=0$, $(f^m)'(\xi_0)<0$.
\end{lemma}
\begin{proof}
Part (b) of \cite[Theorem 5, Section 3.10]{Pe} gives that the flow of the system \eqref{PSsyst1} near the points $Q_2$ and $Q_3$ is topologically equivalent to the flow near the origin $(x,z,w)=(0,0,0)$ of the system
\begin{equation}\label{systinf2}
\left\{\begin{array}{ll}\pm\dot{x}=-mx+x^2-\frac{\beta}{\alpha}xw+x^2w^2-xzw,\\
\pm\dot{z}=-(m+p-1)z-\frac{\beta}{\alpha}zw-(\sigma-1)xz-z^2w+xzw,\\
\pm\dot{w}=-w-\frac{\beta}{\alpha}w^2+xw^2-xw-zw^2,\end{array}\right.
\end{equation}
where the minus sign works for one of the points and the plus sign for the other point. From the second equation of the original system \eqref{PSsyst1} we infer that
$$
\dot{Y}=-Y^2-\frac{\beta}{\alpha}Y+X(1-Y)-Z\sim-Y^2<0
$$
in a neighborhood of both points $Q_2$ and $Q_3$ (since $Y\to\pm\infty$ and dominates over the other variables in a neighborhood of these points), which gives the direction of the flow from right to left and proves that the minus sign in the system \eqref{systinf2} corresponds to $Q_2$ and the plus sign to $Q_3$. Thus $Q_2$ is an unstable node and $Q_3$ is a stable node. In order to establish the local behavior, we notice that in a neighborhood of the origin of the system \eqref{systinf2} we have
$$
\frac{dx}{dw}\sim m\frac{x}{w},
$$
whence by integration $x\sim Cw^m$ or in terms of initial variables $X\sim CY^{1-m}$. Using the formulas for $X$, $Y$ in \eqref{PSvar1} we obtain
$$
(f^m)'(\xi)\sim C\xi^{(m+1)/(m-1)},
$$
and the desired sign-changing behavior at some finite point $\xi=\xi_0\in(0,\infty)$ following a very similar discussion as in \cite[Lemma 2.7]{IS19a} or \cite[Lemma 2.4]{IS20}. We omit the details.
\end{proof}
We next analyze first the critical point $Q_5$ and let $Q_4$ for the end. This is motivated by the fact that the local analysis near $Q_5$ follows the same techniques as used in the previous Lemmas.
\begin{lemma}\label{lem.Q5}
The critical point $Q_5$ in the Poincar\'e hypersphere has a two-dimensional unstable manifold and a one-dimensional stable manifold. The orbits going out from this point into the finite region of the phase space contain profiles satisfying
$$
f(0)=0, \ \quad f(\xi)\sim K\xi^{1/m} \ {\rm as} \  \xi\to0, \ K>0,
$$
in a right-neighborhood of $\xi=0$.
\end{lemma}
\begin{proof}
We infer again from \cite[Section 3.10]{Pe} that the flow in a neighborhood of the point $Q_5$ is topologically equivalent to the flow of the already considered system \eqref{systinf1} but in a neighborhood of the critical point $(y,z,w)=(1/m,0,0)$. Moreover, when approaching $Q_5$ we have
$$
\frac{X}{Y}=\frac{\overline{X}}{\overline{Y}}\sim m,
$$
whence $X\sim mY$ in a (finite) neighborhood of $Q_5$ and
$$
\dot{X}=X[(m-1)Y-2X]\sim-m(m+1)Y^2<0,
$$
thus we have to choose again the minus sign in the system \eqref{systinf1}. The linearization of \eqref{systinf1} near $Q_5$ (including the change of sign given by the minus sign in front of $\dot{y}$, $\dot{z}$, $\dot{w}$) has the matrix
$$
M(Q_5)=\left(
         \begin{array}{ccc}
           -1 & 0 & \frac{m\alpha-\beta}{m\alpha} \\
           0 & \frac{m\sigma+p+1}{m} & 0 \\
           0 & 0 & \frac{m+1}{m} \\
         \end{array}
       \right),
$$
thus we find a two-dimensional unstable manifold and a one-dimensional stable manifold. Analyzing the eigenvectors of the matrix $M(Q_5)$ we find that the orbits going out from $Q_5$ on the unstable manifold go to the finite part of the phase-space, while the orbits entering $Q_5$ on the stable manifold remain on the boundary of the hypersphere. In order to study the profiles contained in the orbits going out of $Q_5$, we deduce from the relation $X\sim mY$ that
$$
\frac{m}{\alpha}f^{m-1}(\xi)\xi^{-2} \sim \frac{m^2}{\alpha}f^{m-2}(\xi)f'(\xi)\xi^{-1},
$$
and after direct integration we obtain
$$
f(\xi)\sim K\xi^{1/m} \ \ {\rm as} \ \xi\to0, \qquad {\rm for} \ K>0,
$$
as desired.
\end{proof}
We remain with the point $Q_4$, that brings nothing new for our analysis. We indeed have
\begin{lemma}\label{lem.Q4}
There are no profiles contained in the orbits connecting to the critical point $Q_4$.
\end{lemma}
\begin{proof}
The point $Q_4$ is characterized by $Z\to\infty$ and $Z/X\to\infty$, $Z/Y\to\infty$ on orbits entering or going out of $Q_4$, that implies in particular that
\begin{equation}\label{interm9}
\xi^{\sigma-2}f(\xi)^{m+p-2}\to\infty, \qquad \xi^{\sigma}f(\xi)^{p-1}\to\infty, \qquad \xi^{1-\sigma}f^{-p}(\xi)f'(\xi)\to0.
\end{equation}
It is obvious that \eqref{interm9} cannot be fulfilled for $\xi\to\xi_0\in(0,\infty)$, as that would mean on the one hand that $f(\xi)\to\infty$ as $\xi\to\xi_0$ and on the other hand that $f(\xi)\to0$ as $\xi\to\xi_0$, since $m+p-2>0$ but $p-1<0$, and a contradiction.

Assume now for contradiction that \eqref{interm9} holds true as $\xi\to0$. If $\sigma\geq2$, thus $\sigma-2\geq0$, we immediately reach a contradiction since \eqref{interm9} implies that at the same time $f^{p-1}(\xi)\to\infty$ and $f^{m+p-2}(\xi)\to\infty$ as $\xi\to0$, and a contradiction. We remain with the case $2(1-p)/(m-1)<\sigma<2$. In that case, we deduce from \eqref{interm9} that in a right-neighborhood of the origin both expressions are larger than 1, thus
$$
\xi^{(2-\sigma)/(m+p-2)}<f(\xi)<\xi^{\sigma/(1-p)}
$$
for $\xi\in(0,\xi_0)$ for some $\xi_0\in(0,1)$. This in particular implies that
$$
\frac{2-\sigma}{m+p-2}>\frac{\sigma}{1-p}
$$
or equivalently $\sigma(m-1)+2(p-1)<0$ and a contradiction with the choice of $\sigma$. Finally, let us assume that \eqref{interm9} holds true as $\xi\to\infty$. Making use of the differential equation for the profiles \eqref{SSODE}, one can prove that the condition
\begin{equation}\label{interm10}
\lim\limits_{\xi\to\infty}\xi^{\sigma}f(\xi)^{p-1}=\infty
\end{equation}
implies that $f$ is strictly decreasing on some interval $(R,\infty)$ for some large $R$ and $f(\xi)\to0$ as $\xi\to\infty$. This has been done in detail in \cite[Lemma 2.8, Steps 1-3]{IS19b} and for the sake of completeness we only sketch here these steps. In a first step, one can show that $f(\xi)$ is monotonic in a neighborhood at infinity. Indeed, supposing that $(\xi_{0,n})_{n\to\infty}$ is an unbounded sequence of local minima for the profile $f$, we readily deduce by evaluating \eqref{SSODE} at $\xi=\xi_{0,n}$ that
$$
\xi_{0,n}^{\sigma}f(\xi_{0,n})^p\leq\alpha f(\xi_{0,n}), \quad n\geq1,
$$
whence $\xi_{0,n}^{\sigma}f(\xi_{0,n})^{p-1}\leq\alpha$, which contradicts \eqref{interm10}. Thus, as no unbounded sequence of local minima exists, $f$ is monotone on some interval $(R,\infty)$ and there exists $L=\lim\limits_{\xi\to\infty}f(\xi)$. It is then easy to discard with the aid of \eqref{SSODE} that $f(\xi)\to\infty$ as $\xi\to\infty$. The possibility that $L\in(0,\infty)$ is also discarded as follows: standard calculus results (see for example \cite[Lemma 2.9]{IL13}) give that there exists a subsequence $\{\xi_n\}_{n\geq1}$ such that
$$
\lim\limits_{n\to\infty}(f^m)''(\xi_n)=\lim\limits_{n\to\infty}\xi_nf'(\xi_n)=0, \quad \lim\limits_{n\to\infty}\xi_n=\infty.
$$
We then find by evaluating then \eqref{SSODE} at $\xi=\xi_n$ that
$$
\lim\limits_{n\to\infty}(\xi_n^{\sigma}f(\xi_n)^{p}-\alpha f(\xi_n))=0,
$$
which is again in contradiction with \eqref{interm10}. We thus remain with the case $f(\xi)\to0$ as $\xi\to\infty$ and strictly decreasing on some interval $(R,\infty)$, $R>0$. Then, if $2(1-p)/(m-1)<\sigma\leq2$ we already get a contradiction with \eqref{interm9}, since in that case
$$
\lim\limits_{\xi\to\infty}\xi^{\sigma-2}f(\xi)^{m+p-2}=0.
$$
If $\sigma>2$, we can write \eqref{SSODE} in the form
\begin{equation}\label{interm11}
(f^m)''(\xi)+\left(\frac{\xi^{\sigma}f^{p-1}(\xi)}{2}-\alpha\right)f(\xi)+\frac{\xi^{\sigma}f^{p-1}(\xi)}{2}+\beta\xi f'(\xi)=0.
\end{equation}
We infer from \eqref{interm9} that there exists $R_0>R$ sufficiently large such that for any $\xi>R_0$ we have
$$
\left(\frac{\xi^{\sigma}f^{p-1}(\xi)}{2}-\alpha\right)f(\xi)>0
$$
and also
$$
\frac{\xi^{\sigma}f^{p-1}(\xi)}{2}+\beta\xi f'(\xi)=\xi^{\sigma}f(\xi)^p\left(\frac{1}{2}+\xi^{1-\sigma}f(\xi)^{-p}f'(\xi)\right)>0.
$$
Since there exists at least a subsequence $(\xi_n)_{n\geq1}$ such that $\xi_n\to\infty$ and $(f^m)''(\xi_n)>0$ for any positive integer $n$, the above inequalities contradict \eqref{interm11} evaluated at $\xi=\xi_n$ for $n$ large enough, ending the proof.
\end{proof}
We close this section with a \emph{local uniqueness} of profiles with interface of Type I. This will allow us to employ the \emph{backward shooting method} to prove the existence of good profiles with interface of Type I in the next section.
\begin{proposition}\label{prop.uniq}
For any $\xi_{0}\in(0,\infty)$ there exists a unique profile (good or not) such that $f(\xi_0)=0$ and $f$ has an interface of Type I at $\xi=\xi_0$.
\end{proposition}
\begin{proof}
For this proof, it is not easy to work with our system \eqref{PSsyst1} since all the profiles with interface of Type I are gathered in the critical point $P_1$. We thus use a different change of variable which identifies the profiles in terms of their interface point by letting
\begin{equation}\label{PSvar2}
x(\eta)=f^{m+p-2}(\xi), \ y(\eta)=(f^{m-2}f')(\xi), \ z(\eta)=\xi, \qquad \frac{d\eta}{d\xi}=mf^{m-1}(\xi),
\end{equation}
thus obtaining the following system
\begin{equation}\label{PSsyst2}
\left\{\begin{array}{ll}\dot{x}=m(m+p-2)xy,\\
\dot{y}=-my^2-\beta yz+\alpha x^{(m-1)/(m+p-2)}-|z|^{\sigma}x,\\
\dot{z}=mx^{(m-1)/(m+p-2)}.\end{array}\right.
\end{equation}
Since $m-1>m+p-2>0$, it is easy to check that the points with behavior interface of Type I are identified as the critical line $my+\beta z=0$ in the invariant plane $\{x=0\}$, that is, the critical points of coordinates $P(\xi_0)=(0,-\beta\xi_0/m,\xi_0)$ for $\xi_0\in(0,\infty)$ given. The linearization of the system \eqref{PSsyst2} in a neighborhood of this critical point has the matrix
$$M(\xi_0)=\left(
  \begin{array}{ccc}
    -(m+p-2)\beta\xi_0 & 0 & 0 \\
    -\xi_0^{\sigma} & \beta\xi_0 & \frac{\beta^2\xi_0}{m} \\
    0 & 0 & 0 \\
  \end{array}
\right),
$$
with eigenvalues $\lambda_1=-(m+p-2)\beta\xi_0<0$, $\lambda_2=\beta\xi_0>0$ and $\lambda_3=0$. We infer from \cite[Theorem 2.15, Chapter 9]{CH} and the Local Center Manifold Theorem \cite[Theorem 1, Section 2.10]{Pe} that all the center manifolds (recall that the center manifold may not be unique) of dimension one in the neighborhood of $P(\xi_0)$ have to contain a segment of the invariant line $\{x=0, my+\beta z=0\}$. We thus readily deduce that the center manifold near $P(\xi_0)$ is unique and by well-known results also the one-dimensional stable and unstable manifolds are unique. Similarly to the analysis in \cite[Lemma 2.2]{IS19a} (see also \cite{dPS02} for more details), there exists only one orbit entering $P(\xi_0)$ from outside the invariant plane $\{x=0\}$. All the other orbits are contained in the plane $\{x=0\}$ and do not contain profiles. We thus obtain the desired uniqueness.
\end{proof}

\section{Existence of good profiles with interface of Type I}\label{sec.type1}

In this section, we employ the local analysis performed in the previous sections to show that for any $\sigma>2(1-p)/(m-1)$ there exists at least one good profile with interface of Type I. Since the same fact for profiles with interface of Type II has been proved in Section \ref{sec.type2}, this completes the proof of Theorem \ref{th.exist}. The strategy used to prove this existence result is the \emph{backward shooting method}, that is shooting from the interface point $\xi=\xi_0\in(0,\infty)$ and trace backward the unique profile with interface of Type I exactly at $\xi=\xi_0$, according to Proposition \ref{prop.uniq}. The idea is to show that profiles with an interface at $\xi_0>0$ very small are strictly decreasing, while profiles with an interface at $\xi_0$ large have a change of sign at some point $\xi_1\in(0,\xi_0)$. However, because of techical reasons we cannot perform the backward shooting in the phase space associated to the system \eqref{PSsyst1} and we introduce a new change of variables by setting $Z=UV$, $X=U^{(m-1)/(m+p-2)}$ or equivalently
\begin{equation}\label{PSvar3}
\begin{split}
&U=X^{(m+p-2)/(m-1)}=\left(\frac{m}{\alpha}\right)^{(m+p-2)/(m-1)}\xi^{-2(m+p-2)/(m-1)}f(\xi)^{m+p-2}, \\
&V=\frac{Z}{U}=\frac{1}{\alpha}\left(\frac{m}{\alpha}\right)^{(1-p)/(m-1)}\xi^{[\sigma(m-1)+2(p-1)]/(m-1)}.
\end{split}
\end{equation}
In variables $(U,Y,V)$ we obtain the autonomous dynamical system
\begin{equation}\label{PSsyst3}
\left\{\begin{array}{ll}\dot{U}=\frac{m+p-2}{m-1}U[(m-1)Y-2U^{(m-1)/(m+p-2)}],\\
\dot{Y}=-Y^2-\frac{\beta}{\alpha}Y+U^{(m-1)/(m+p-2)}(1-Y)-UV,\\
\dot{V}=\frac{\sigma(m-1)+2(p-1)}{m-1}U^{(m-1)/(m+p-2)}V,\end{array}\right.
\end{equation}
and we notice that, despite the fact that the system \eqref{PSsyst3} is no longer quadratic, it has a very important property that the third equation is very simple and \emph{the component $V$ is non-decreasing} along the trajectories in the phase space. We moreover notice that $V$ is a power of $\xi$ and the behavior of interface of Type I means now orbits entering the critical points $P(v_0)=(0,-\beta/\alpha,v_0)$ with $v_0\geq0$. The uniqueness proved in Proposition \ref{prop.uniq} can be easily transferred here and thus get that for every $v_0>0$ there exists a unique orbit entering the critical point $P(v_0)$ coming from the interior of the phase space and containing the unique profile with interface at the point $\xi=\xi_0\in(0,\infty)$ given by
\begin{equation}\label{interm14}
v_0=\frac{1}{\alpha}\left(\frac{m}{\alpha}\right)^{(1-p)/(m-1)}\xi_0^{[\sigma(m-1)+2(p-1)]/(m-1)}.
\end{equation}
We are thus ready to start our backward shooting method, which is formalized in the two propositions below.
\begin{proposition}\label{prop.close}
In the previous notation, the orbits entering points $P(v_0)$ with $v_0>0$ sufficiently small contain profiles $f(\xi)$ that are decreasing and have a negative slope at $\xi=0$, that is, $f(0)=a>0$, $f'(0)<0$.
\end{proposition}
\begin{proof}
Since any profile with interface is decreasing in a neighborhood of the interface point, recalling that at any point $P(v_0)$ we have $Y=-\beta/\alpha$, a non-decreasing profile must cross first the plane $\{Y=0\}$ in the phase space associated to the system \eqref{PSsyst3} and then also the plane $\{Y=-\beta/2\alpha\}$ before reaching any of the critical points $P(v_0)$. The direction of the flow on the plane $\{Y=-\beta/2\alpha\}$ is given by the sign of the expression
\begin{equation*}
\begin{split}
F(U,V)&=-\frac{\beta^2}{4\alpha^2}+\frac{\beta^2}{2\alpha^2}+U^{(m-1)/(m+p-2)}\left(1+\frac{\beta}{2\alpha}\right)-UV\\
&=U\left[\left(1+\frac{\beta}{2\alpha}\right)U^{(1-p)/(m+p-2)}-V\right]+\frac{\beta^2}{4\alpha^2}.
\end{split}
\end{equation*}
The orbits crossing this plane have to do it in the region where $F(U,V)<0$, which is equivalent to
\begin{equation}\label{interm13}
V\geq h(U):=\left(1+\frac{\beta}{2\alpha}\right)U^{(1-p)/(m+p-2)}+\frac{\beta^2}{4U\alpha^2}.
\end{equation}
One can readily optimize in $U$ in the expression in \eqref{interm13} to find that $h(U)$ has a positive minimum $\overline{v_0}=h(U_0)$ attained at
$$
U_0=\left[\frac{(m+p-2)(m-p)^2}{2(\sigma+2)(2\sigma+4+m-p)(1-p)}\right]^{(m+p-2)/(m-1)}
$$
Since the variable $V$ is monotone increasing along the trajectories, it follows that an orbit crossing the plane $\{Y=-\beta/2\alpha\}$ can reach critical points $P(v_0)$ only for $v_0>\overline{v_0}=h(U_0)$. Thus the profiles contained in the orbits entering the points $P(v_0)$ with $v_0\leq h(U_0)$ are  decreasing.
\end{proof}

\noindent \textbf{Remark.} Let $f$ be such a decreasing profile (as obtained in Proposition \ref{prop.close} for $v_0$ small). Then the self-similar formula
$$
u(x,t)=(T-t)^{-\alpha}f(|x|(T-t)^{\beta}), \qquad T>0, \ t\in(0,T),
$$
gives a one-parameter family of \textbf{supersolutions} to Eq. \eqref{eq1} for any such fixed profile $f(\xi)$. We stress here that these supersolutions will be strongly used for comparison in the forthcoming paper \cite{IMS20} in order to prove the local existence and finite speed of propagation of general solutions to a similar equation to \eqref{eq1} with compactly supported initial conditions in the range $m+p>2$.

\medskip

\noindent With respect to shooting from $\xi_0\in(0,\infty)$ very large we state
\begin{proposition}\label{prop.far}
In the previous notation, the orbits entering points $P(v_0)$ with $v_0>0$ sufficiently large contain profiles $f(\xi)$ with a \emph{backward change of sign} at some point $\xi_1\in(0,\xi_0)$ in the following sense
$$
f(\xi_1)=0, \qquad (f^m)'(\xi_1)>0, \qquad f(\xi)>0 \ {\rm for } \ \xi\in(\xi_1,\xi_0),
$$
where $xi_0$ and $v_0$ are related by \eqref{interm14}.
\end{proposition}
\begin{proof}
First of all, we work in the invariant plane $\{X=0\}$ seen as a limiting case in variables $(X,Y,Z)$. The phase space associated to the system \eqref{PSsyst1} restricts to the following system
\begin{equation}\label{interm15}
\left\{\begin{array}{ll}\dot{Y}=-Y^2-\frac{\beta}{\alpha}Y-Z,\\
\dot{Z}=(m+p-2)YZ,\end{array}\right.
\end{equation}
which has been considered (with a difference only at the level of a constant) in \cite[Proposition 3.2]{IS19a} and \cite[Proposition 3.4, Step 1]{IS19b}. It is shown that in the plane $\{X=0\}$ there exists a unique orbit entering the critical point $(0,-\beta/\alpha)$ which is a saddle point for the system \eqref{interm15}, and this orbit comes from the unstable node $Q_2$ at infinity. Let $R(z_0)$ be a point on this unique orbit inside the plane $\{X=0\}$ and with component $Z=z_0>0$. By Lemma 3.1 and its proof we deduce that this unique orbit contained in $\{X=0\}$ enters $P_1$ tangent to the eigenvector $e_3=(0,1,(m+p-1)\beta/\alpha)$, thus if $z_0>0$ is taken to be sufficiently small, we readily get that
$$
R(z_0)=\left(0,-\frac{\beta}{\alpha}+\frac{z_0\alpha}{(m+p-1)\beta}+o(z_0),z_0\right)
$$
We consider small balls $B(R(z_0),\delta)$ centered at $R(z_0)$. The next step in the proof is to show that for any given radius $\delta>0$, there exists some $v(\delta)$ sufficiently large such that the unique orbit entering the critical point $P(v_0)=(0,-\beta/\alpha,v_0)$ for any $v_0>v(\delta)$ in the phase space associated to the system \eqref{PSsyst3} intersects $B(R(z_0),\delta)$. To this end, we fix $v_0>0$ and perform the change of variable in \eqref{PSsyst3}
$$
H=Y+\frac{\beta}{\alpha}, \qquad \overline{V}=V-v_0
$$
which maps $P(v_0)$ into the origin of a new system in variables $(U,H,\overline{V})$. The fact that $(m-1)/(m+p-2)>1$ and easy calculations give that the projection of the orbit entering $P(v_0)$ onto the plane $\{\overline{V}=0\}$ satisfies the following system in a neighborhood of $(U,H,\overline{V})=(0,0,0)$:
\begin{equation}\label{interm12}
\left\{\begin{array}{ll}\dot{U}=-\frac{\beta(m+p-2)}{\alpha}U+o(|(U,H)|),\\
\dot{H}=\frac{\beta}{\alpha}H-Uv_0+o(|(U,H)|),\end{array}\right.
\end{equation}
whose linearization has explicit trajectories obtained by an easy integration to get
$$
H=\frac{\alpha v_0}{\beta(m+p-1)}U, \ {\rm or \ equivalently} \ Y=-\frac{\beta}{\alpha}+\frac{\alpha v_0}{\beta(m+p-1)}U.
$$
The trajectories of the nonlinear system \eqref{interm12} are approximated by the linear ones above. It thus follows that in a neighborhood of $P(v_0)$ the points on the trajectory entering $P(v_0)$ have the form
$$
Q(\lambda)=\left(\lambda,-\frac{\beta}{\alpha}+\frac{\alpha v_0}{\beta(m+p-1)}\lambda+o(\lambda),v_0+o(\lambda)\right)
$$
for $\lambda>0$ sufficiently small. Coming back to the initial variables $(X,Y,Z)$ by undoing the change of variable \eqref{PSvar3}, the above points become
$$
Q(\lambda)=\left(\lambda^{(m+p-2)/(m-1)},-\frac{\beta}{\alpha}+\frac{\alpha v_0}{\beta(m+p-1)}\lambda+o(\lambda),\lambda v_0+o(\lambda^2)\right).
$$
Letting now $\lambda=z_0/v_0$ we get that the previous trajectory passes through points of the form
$$
Q(z_0)=\left(\left(\frac{z_0}{v_0}\right)^{(m+p-2)/(m-1)},-\frac{\beta}{\alpha}+\frac{\alpha z_0}{\beta(m+p-1)}+o(z_0),z_0+o(z_0^2)\right)
$$
and given $\delta>0$, there exists a sufficiently large $v(\delta)$ such that $Q(z_0)\in B(R(z_0),\delta)$ for any $v_0>v(\delta)$. We end the proof by a standard continuity argument showing, since $Q_2$ is an unstable node, that there exists $\delta_0>0$ sufficiently small such that all the trajectories intersecting the ball $B(R(z_0),\delta_0)$ come from $Q_2$, and in particular also come from $Q_2$ all the orbits entering $P(v_0)$ for $v_0>v(\delta_0)$. \end{proof}
The proof of Theorem \ref{th.exist} for profiles with interface of Type I is now standard and we will just give a sketch.
\begin{proof}[Proof of Theorem \ref{th.exist}]
Let $A\subseteq(0,\infty)$ be the set of points $\eta_0\in(0,\infty)$ such that the unique profile having an interface of Type I at $\xi=\eta_0$ (according to Proposition \ref{prop.uniq}) intersects the vertical axis with negative slope, that is, $f(0)=a>0$, $f'(0)<0$. It follows by a standard argument of continuity that $A$ is an open set which is nonempty according to Proposition \ref{prop.close}. Let then $\xi_0=\sup A$. Thus, $\xi_0\not\in A$ (since $A$ is open) and $\xi_0<\infty$, as it readily follows from Proposition \ref{prop.far}. It is then easy to check that the profile having an interface of Type I exactly at $\xi=\xi_0$ is a good profile with interface of Type I. We refer the reader to \cite[Section 3]{IS19a} for a detailed proof of this statement, which applies absolutely identically in the present case.
\end{proof}

\section{Blow-up profiles for $\sigma$ small}\label{sec.small}

This section is devoted to the \emph{proof of part (a) in Theorem \ref{th.small}}. Let us stress first that by $\sigma$ small we understand in this case $\sigma$ sufficiently close to its lower limit $2(1-p)/(m-1)$ and not to 0, as in \cite{IS19a}. We begin with the following
\begin{proposition}\label{prop.small}
There exists $\sigma_0>2(1-p)/(m-1)$ such that for any $\sigma\in(2(1-p)/(m-1),\sigma_0)$, all the orbits going out from the points $P_0$ and $P_2$ into the interior of the phase space associated to the system \eqref{PSsyst1} connect to the point $P_0$. Thus, all the profiles contained in these orbits are good blow-up profiles with interface of Type II.
\end{proposition}
\begin{proof}
Although the proposition is stated in terms of the system in variables $(X,Y,Z)$, we prove it using once more the new variables $(U,Y,V)$ introduced in \eqref{PSvar3} and the autonomous system \eqref{PSsyst3}. Borrowing the plan of the proof from \cite[Proposition 4.1]{IS19a}, the general plan is to "trace" the unique orbit going out of $P_2$ (according to Lemma \ref{lem.P2}) by imposing suitable barriers for it. Let us notice first that in variables $(U,Y,V)$
$$
P_2=\left(\left(\frac{m-1}{2(m+1)\alpha}\right)^{(m+p-2)/(m-1)},\frac{1}{(m+1)\alpha},0\right)=(U(P_2),Y(P_2),0),
$$
in order to shorten the notation. We divide the proof into several steps.

\medskip

\noindent \textbf{Step 1.} On the one hand, the direction of the flow of the system \eqref{PSsyst3} on the plane $\{U=U(P_2)\}$ is given by the sign of the expression
$$
\frac{m+p-2}{m-1}U\left[(m-1)Y-\frac{m-1}{(m+1)\alpha}\right]=(m+p-2)U[Y-Y(P_2)],
$$
which is negative for $Y<Y(P_2)$. On the other hand, the direction of the flow of the system \eqref{PSsyst3} on the plane $\{Y=Y(P_2)\}$ is given by the sign of the expression
\begin{equation*}
\begin{split}
-Y(P_2)^2&-\frac{\beta}{\alpha}Y(P_2)+U^{(m-1)/(m+p-2)}(1-Y(P_2))-UV\\&<U^{(m-1)/(m+p-2)}(1-Y(P_2))-\frac{(m+1)\beta+1}{(m+1)^2\alpha^2},
\end{split}
\end{equation*}
and the latter is negative for $U<U(P_2)$. Since the connection going out of $P_2$ is tangent to the eigenvector $e_3$ in Lemma \ref{lem.P2} having negative $X$ and $Y$ component, it follows that this connection goes out from $P_2$ in the region $\{U<U(P_2), Y<Y(P_2)\}$ and it remains forever in this region according to the direction of the flow. Moreover, all the connections going out of $P_0$ enter the same region.

\medskip

\noindent \textbf{Step 2.} We next look for a constant $k>0$ such that the plane of equation $\{Y+kV=1\}$ be an upper barrier for the orbits from $P_2$ and $P_0$. The direction of the flow of the system \eqref{PSsyst3} over this plane is given by the sign of the expression
\begin{equation*}
\begin{split}
F(U,Y,V)&=-Y^2-\frac{\beta}{\alpha}Y+kU^{(m-1)/(m+p-2)}V\\&+k\frac{\sigma(m-1)+2(p-1)}{m-1}U^{(m-1)/(m+p-2)}V-UV\\
&=-Y^2-\frac{\beta}{\alpha}Y+UV\left[k\frac{(\sigma+1)(m-1)+2(p-1)}{m-1}U^{(1-p)/(m+p-2)}-1\right],
\end{split}
\end{equation*}
which, taking into account that along the orbits we are considering we have $U<U(P_2)$, is negative for $Y<0$ if we take for example $k$ such that
\begin{equation}\label{interm16}
\frac{1}{k}=\frac{(\sigma+1)(m-1)+2(p-1)}{m-1}U(P_2)^{(1-p)/(m+p-2)}.
\end{equation}
Thus, in the region $Y\geq0$ the orbits starting from $P_2$ and $P_0$ satisfy the bound $Y+kV\leq1$ with $k$ as in \eqref{interm16}. In particular, the orbits will intersect the plane $\{Y=0\}$ at a point whose coordinate $V$ fulfills $V\leq 1/k$. Thus, at this crossing point, we have
\begin{equation}\label{interm17}
UV\leq\frac{U(P_2)}{k}=\frac{[\sigma(m-1)+2(p-1)][(\sigma+1)(m-1)+2(p-1)]}{2(\sigma+2)(m+1)}\to0 \ {\rm as} \ \sigma\to\frac{2(1-p)}{m-1}.
\end{equation}
Letting
$$
k_1:=\frac{\beta^2}{4\alpha^2}=\frac{(m-p)^2}{4(\sigma+2)^2},
$$
we infer from \eqref{interm17} that there exists $\sigma_0$ sufficiently small (that we can take to also be smaller than 2) such that at the intersection point with the plane $\{Y=0\}$ the orbits satisfy
$$
UV<k_1, \qquad {\rm for \ any} \ \sigma\in\left(\frac{2(1-p)}{m-1},\sigma_0\right),
$$
thus they enter the half-space $\{Y<0\}$ in the region lying below the hyperbolic cylinder $\{UV=k_1\}$.

\medskip

\noindent \textbf{Step 3.} The direction of the flow on the plane $\{Y=0\}$ is given by the sign of the expression
$$
h(U,V)=U\left(U^{(1-p)/(m+p-2)}-V\right).
$$
Thus, the plane $\{Y=0\}$ can be crossed from right to left in the region where $h(U,V)<0$. By inspecting the equations for $\dot{U}$ and $\dot{V}$ in the system \eqref{PSsyst3} we deduce that $V$ is increasing, while $U$ is decreasing along the trajectories in the half-space $\{Y\leq0\}$. Thus after the first crossing, $h(U,V)$ remains always negative along the trajectories (as $U$ continues to decrease while $V$ continues to increase). This implies that the orbit will remain forever in the region $\{Y\leq0\}$.

\medskip

\noindent \textbf{Step 4.} We analyze now the direction of the flow of the system over the hyperbolic cylinder $\{UV=k_1\}$ obtained in Step 2. This is given by the sign of the expression
\begin{equation*}
\begin{split}
G(U,Y,V)&=\frac{m+p-2}{m-1}UV[(m-1)Y-2U^{(m-1)/(m+p-2)}]\\&+\frac{\sigma(m-1)+2(p-1)}{m-1}U^{(m-1)/(m+p-2)}UV\\
&=UV\left[(m+p-2)Y+(\sigma-2)U^{(m-1)/(m+p-2)}\right]<0,
\end{split}
\end{equation*}
in the region $\{Y\leq0\}$ and for $\sigma\in(2(1-p)/(m-1),\sigma_0)$, where $\sigma_0$ has been chosen such that $\sigma_0<2$ and as in Step 2. Thus, a connection entering the interior of $\{UV<k_1\}$ for such a $\sigma$, cannot go out from the hyperbolic cylinder.

\medskip

\noindent \textbf{Step 5.} Let us take now as barrier the plane $\{Y=-\beta/2\alpha\}$. The direction of the flow on this plane is given by the sign of
$$
\frac{\beta^2}{4\alpha^2}+\left(1+\frac{\beta}{2\alpha}\right)U^{(m-1)/(m+p-2)}>k_1-UV\geq0,
$$
and we infer that this plane cannot be crossed from right to left by any trajectory through the region $\{UV\leq k_1\}$.

\medskip

\noindent \textbf{Step 6. End of the proof.} Gathering all the previous steps, we notice that for any $\sigma\in(2(1-p)/(m-1),\sigma_0)$, the unique orbit going out of $P_2$ and \textbf{all} the orbits going out of $P_0$ stay forever in the region $\{U<U(P_2), Y<Y(P_2)\}$, and due to Steps 3 and 4, they cross the plane $\{Y=0\}$ in the region where $\{UV<k_1\}$ and thus remain forever in this region. Consequently, as shown in Steps 3 and 5 all these orbits will remain forever also in the strip $\{-\beta/2\alpha\leq Y\leq 0\}$. Since the coordinates $U$ and $V$ are monotonic in the region $\{Y\leq0\}$ along the trajectories, the orbits cannot end in a limit cycle and have to enter a critical point. We infer from the analysis done in Section \ref{sec.type2} and Lemma \ref{lem.Q4} that these orbits have to enter the critical point $P_0$ (in the way explained in Section \ref{sec.type2}) and contain good profiles with interface of Type II.
\end{proof}
The proof of Theorem \ref{th.small}, part (a) is now immediate. Indeed, Proposition \ref{prop.small} proves that there exists $\sigma_0>0$ such that for any $\sigma\in(2(1-p)/(m-1),\sigma_0)$, all the good profiles satisfying property (P2) in Definition \ref{def1} have an interface of Type II. On the other hand, Theorem \ref{th.exist} shows that for any such $\sigma$ there exists also at least a good profile with interface of Type I, and necessarily this good profile satisfies assumption (P1) in Definition \ref{def1}, that is, $f(0)=A>0$, $f'(0)=0$, as stated. We plot in Figure \ref{fig2} a numerical simulation of the behavior of the orbits going out of the critical points $P_2$ and $P_0$ for $\sigma$ sufficiently small (within the range of application of Theorem \ref{th.small}, part (a)).

\begin{figure}[ht!]
  \begin{center}
  \includegraphics[width=11cm,height=7.5cm]{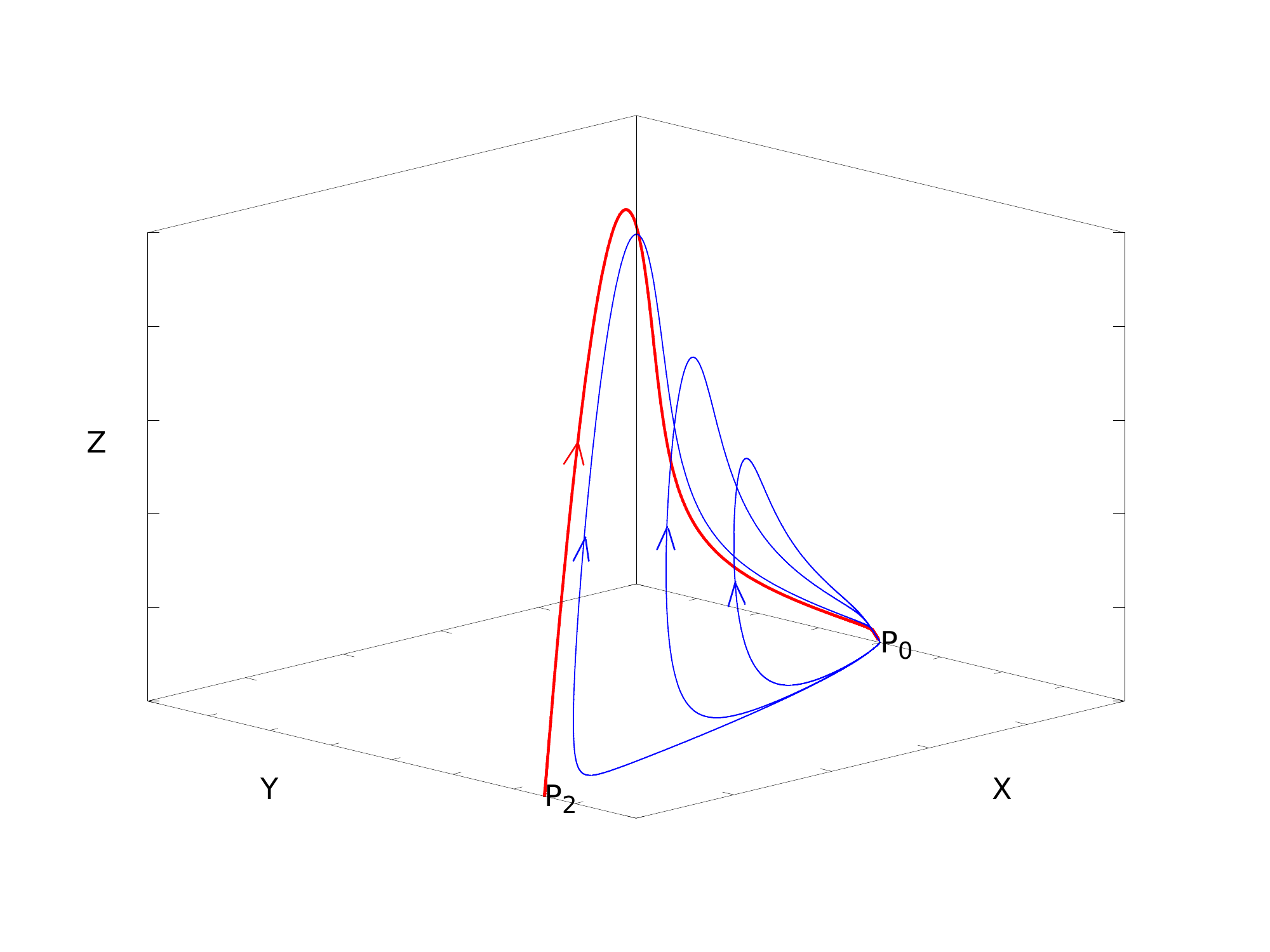}
  \end{center}
  \caption{Orbits going out of $P_2$ and $P_0$ for $\sigma$ small. Numerical experiment for $m=3$, $p=0.5$ and $\sigma=3$}\label{fig2}
\end{figure}

\section{Blow-up profiles with interface for $\sigma$ large}

This section is devoted to the proof of the remaining results for the range $m+p>2$, that is, part (b) in Theorem \ref{th.small} and Theorem \ref{th.large}. The core of the argument is to prove that for $\sigma$ sufficiently large, the connection going out of $P_2$ according to Lemma \ref{lem.P2} enters the critical point $Q_3$ in the phase space associated to the system \eqref{PSsyst1}. All these proofs are very similar to the ones in \cite[Section 5]{IS19b} and we will give a sketch of them or quote them directly if no differences appear. We start with the following technical result:
\begin{lemma}\label{lem.monot}
Let $\sigma\geq2$. Then the component $X$ is decreasing and the component $Y$ is also decreasing in the half-space $\{Y\geq0\}$ along the trajectory going out of the point $P_2$.
\end{lemma}
\begin{proof}
It is obvious from the equation for $\dot{X}$ in the system \eqref{PSsyst1} that $X$ decreases along any trajectory in the region $\{Y<0\}$. Assume for contradiction that the coordinate $X$ is not decreasing along the orbit going of $P_2$ (necessarily this should happen for $Y\geq0$). Since both components $X$, $Y$ start in a decreasing way in a neighborhood of $P_2$, there exists a first point $\eta_1>0$ such that $\dot{X}(\eta_1)=0$, $X''(\eta_1)\geq0$. But
$$
0\leq X''(\eta_1)=(m-1)X(\eta_1)\dot{Y}(\eta_1),
$$
whence $\dot{Y}(\eta_1)\geq0$. Thus coordinate $Y$ had to change monotonicity already along the trajectory going out of $P_2$ at some first point $\eta_2\leq\eta_1$. That means $\dot{Y}(\eta_2)=0$ and $Y''(\eta_2)\geq0$. If $\eta_2=\eta_1$, since $\dot{X}(\eta_2)=\dot{Y}(\eta_2)$ we obtain after differentiating again the second equation in \eqref{PSsyst1} that
$$
Y''(\eta_2)=-\dot{Z}(\eta_2)=-Z(\eta_2)[(m+p-2)Y(\eta_2)+(\sigma-2)X(\eta_2)]<0
$$
and a contradiction. If $\eta_2\in(0,\eta_1)$, that means $\dot{X}(\eta_2)<0$ (since $\eta_1>\eta_2$ is the first point where $X$ ceases to be decreasing). Taking into account that along the orbit going out of $P_2$ we have $Y\leq Y(P_2)\leq1$ and that for $\sigma>2$ $Z$ is increasing in the region $\{Y\geq0\}$ we derive again that
$$
Y''(\eta_2)=\dot{X}(\eta_2)(1-Y(\eta_2))-\dot{Z}(\eta_2)<0,
$$
and a contradiction. Thus $X$ is decreasing and one can check in a similar way that the component $Y$ is also decreasing along the orbit from $P_2$.
\end{proof}
Coming back to the analysis of the invariant plane $\{Z=0\}$, we have the following preparatory result which has been already proved as \cite[Lemma 5.4]{IS19a} (to which we refer the interested reader).
\begin{lemma}\label{lem.Z0}
For any $\sigma>2(1-p)/(m-1)$ there exists an orbit connecting the critical points $P_0$ and $P_2$ and included in the invariant plane $\{Z=0\}$.
\end{lemma}
We are now in a position to state the main technical result of this section
\begin{proposition}\label{prop.large}
There exists $\sigma_1>0$ sufficiently large such that for any $\sigma\in(\sigma_1,\infty)$, the unique orbit going out from $P_2$ in the phase space associated to the system \eqref{PSsyst1} enters the critical point $Q_3$. Moreover, for any $\sigma\in(\sigma_1,\infty)$ there are also orbits connecting from $P_0$ to $Q_3$.
\end{proposition}
\begin{proof}
The system \eqref{PSsyst1} is topologically equivalent to the system
\begin{equation}\label{PSsyst4}
\left\{\begin{array}{ll}\dot{\overline{X}}=m\overline{X}[(m-1)\overline{Y}-2\overline{X}],\\
\dot{\overline{Y}}=-m\overline{Y}^2-\beta\overline{Y}+\alpha\overline{X}-m\overline{X}\overline{Y}-\overline{X}\overline{Z},\\
\dot{\overline{Z}}=m\overline{Z}[(p-1)\overline{Y}+\sigma\overline{X}].\end{array}\right.
\end{equation}
obtained for the variables (notice that modulo some constants, $\overline{Z}=Z/X$)
\begin{equation}\label{PSchange4}
\overline{X}(\eta)=\xi^{-2}f^{m-1}(\xi), \ \
\overline{Y}(\eta)=\xi^{-1}f^{m-2}(\xi)f'(\xi), \ \
\overline{Z}(\eta)=\xi^{\sigma}f^{p-1}(\xi),
\end{equation}
and which was used all along the paper \cite{IS19b}. In our case, we could not use this system from the beginning as some critical points become points at infinity since $p<1$. But we can use this system in the current proof, which is now perfectly identical to the proof of \cite[Proposition 5.6]{IS19b}. Let us notice that Lemma \ref{lem.monot} is independent of the change of variable $\overline{Z}=Z/X$, thus it also applies to the system \eqref{PSsyst4}. A careful inspection of the proof of \cite[Proposition 5.6]{IS19b} (and its previous technical result \cite[Lemma 5.5]{IS19b}) which uses Lemma \ref{lem.monot} as an important technical tool, shows that the fact that $p>1$ it is nowhere used along the proof, thus it can be extended to our case. Indeed, the only elements used in an essential way in the proof are the facts that $m>1$, $m>p$, $\sigma(m-1)+2(p-1)>0$, and the fact that, considering the planes
\begin{equation}\label{plane1}
\overline{Z}=E-D\overline{Y}, \quad D=\frac{2m(m+1)^2}{m-1}, \ E=\frac{2(m+1)}{m-1},
\end{equation}
respectively
\begin{equation}\label{plane2}
\overline{X}=B\overline{Y}+C, \quad B=\frac{m(m-1)}{2m^2+5m+1}, \ C=\frac{(2m+1)(m-1)}{2m(2m^2+5m+1)}
\end{equation}
the orbit going out of $P_2$ starts in the region where simultaneously $\overline{Z}>E-D\overline{Y}$ and $\overline{X}>B\overline{Y}+C$. These two facts are also true in our range of parameters, since $\overline{X}$ and $\overline{Y}$ are exactly the same (they do not depend on $p$) and $\overline{Z}$ passes to be at infinity at the starting point of the orbit for $p<1$, thus $\overline{Z}>E-D\overline{Y}$ holds true in a trivial way at the beginning of the orbit. We refer the reader then to the (completely detailed) proof of \cite[Proposition 5.6]{IS19b} for the rather tedious and long calculations showing that the orbit starting from $P_2$ has to cross a critical plane
$$
\overline{Y}=-\overline{Y}_0, \qquad \overline{Y}_0=\frac{(m-1)(\sigma+2)}{2m[\sigma(m-1)+2(p-1)]},
$$
after which it can no longer return. Going back to our initial system \eqref{PSsyst1} and translating the result, we conclude that the connection from $P_2$ will connect to the critical point $Q_3$ for $\sigma$ sufficiently large. Using Lemma \ref{lem.Z0} and standard continuity arguments it follows that for any $\sigma$ large when the orbit from $P_2$ enters $Q_3$, there are also orbits going out of $P_0$ and connecting $Q_3$. The details of this last argument are given in \cite[Proposition 5.6, Step 4]{IS19b} or \cite[pp. 2091-2092]{IS19a}.
\end{proof}

%
With all the previous technical steps, we are in a position to prove part (b) in Theorem \ref{th.small}.
\begin{proof}[Proof of Theorem \ref{th.small}, part (b)] We use the "three-sets argument". Let us consider then the sets
\begin{equation*}
\begin{split}
&A:=\{\sigma>0: {\rm the \ orbit \ from \ }P_2 \ {\rm enters} \ P_0\},\\
&B:=\{\sigma>0: {\rm the \ orbit \ from \ }P_2 \ {\rm enters} \ P_1\},\\
&C:=\{\sigma>0: {\rm the \ orbit \ from \ }P_2 \ {\rm enters} \ Q_3\}.
\end{split}
\end{equation*}
The set $C$ is open since $Q_3$ is an attractor. The same argument cannot be used directly for the point $P_0$, as it is not an attractor by itself. But we get from the analysis in Section \ref{sec.type2} and the classification in \cite{Date79} that there exists a sufficiently small neighborhood $B(P_0,\delta)$ of $P_0$ such that any trajectory of the system entering $B(P_0,\delta)$ either comes out from or enters $P_0$. Since all the connections going out of $P_0$ enter the half-space $\{Y>0\}$, it follows that $P_0$ behaves exactly like an attractor in the half-ball $B(P_0,\delta)\cap\{Y<0\}$, that is, any trajectory entering the half-ball enters $P_0$ afterwards. Since the orbits going out of $P_2$ can only enter $P_0$ after crossing the plane $\{Y=0\}$ (which in terms of profiles means arriving to a maximum point and then starting to decrease towards the interface), these orbits can only enter $P_0$ from the negative side $\{Y<0\}$, thus the same argument as for an attractor shows that $A$ is an open set. Since both $A$ and $C$ are nonempty, as insured by Propositions \ref{prop.small}, respectively \ref{prop.large}, we infer that the set $B$ is also nonempty (and closed), thus there exists at least one $\sigma_*\in(2(1-p)/(m-1),\infty)$ such that $P_2$ connects to $P_1$ for $\sigma=\sigma_*$ (containing thus a profile with interface of Type I).
\end{proof}
The proof of the remaining Theorem \ref{th.large} is similar as the previous "three-sets argument", since the same argument stays true also for the orbits going out of $P_0$ itself: if they enter $P_0$ forming the elliptic sector as shown in Section \ref{sec.type2}, they do that also through the half-space $\{Y<0\}$. We omit the details which are easy and similar to the proof of \cite[Lemma 5.5]{IS19a}. We plot in Figure \ref{fig3} the outcome of numerical experiments on the behavior of the orbits going out of the critical points $P_2$ and $P_0$ in the critical case (as in Theorem \ref{th.small}, part (b)) and for $\sigma$ large (according to Theorem \ref{th.large}).

\begin{figure}[ht!]
  \begin{center}
  \subfigure[Critical $\sigma^*$]{\includegraphics[width=7.5cm,height=6cm]{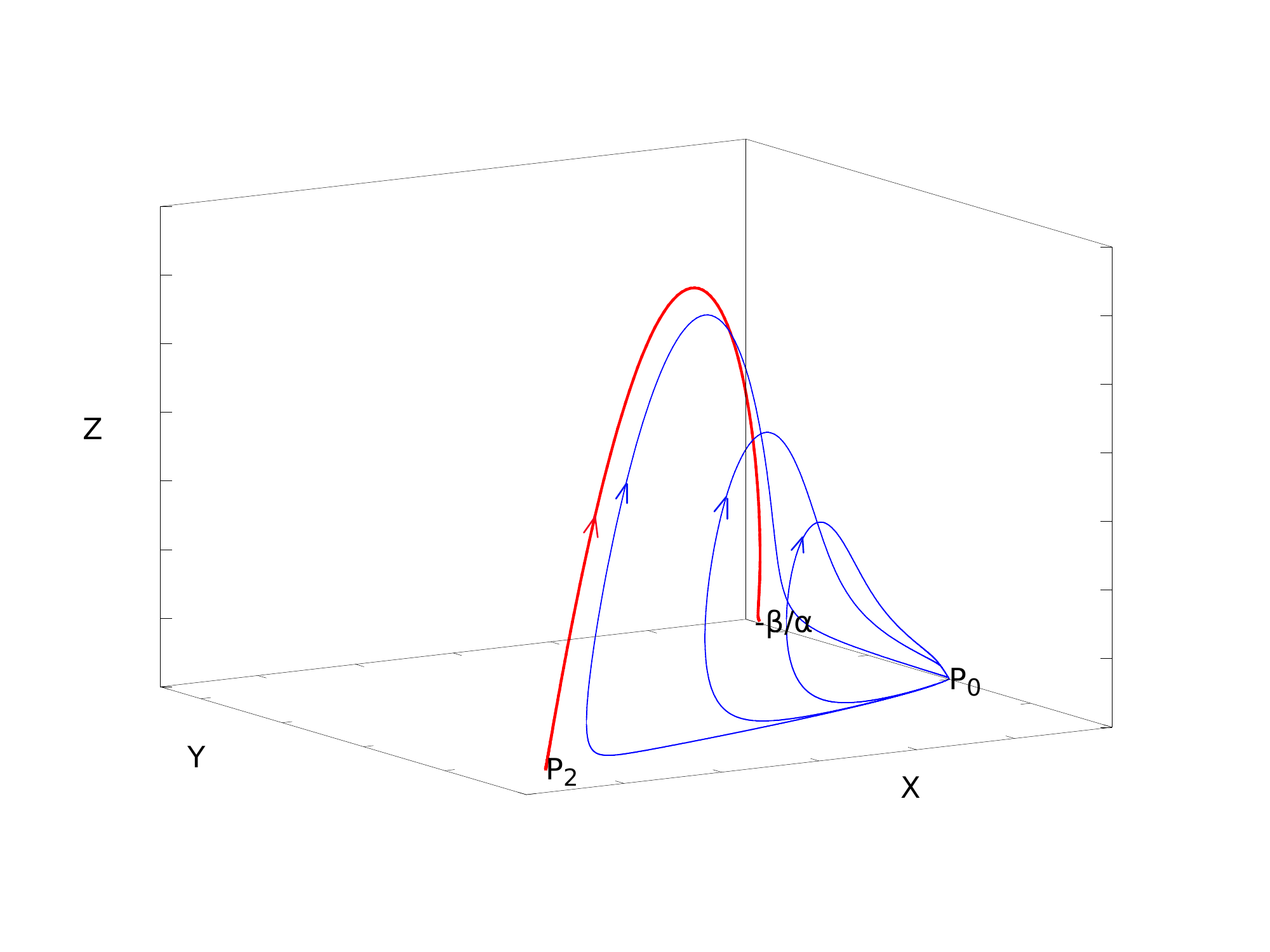}}
  \subfigure[$\sigma$ large]{\includegraphics[width=7.5cm,height=6cm]{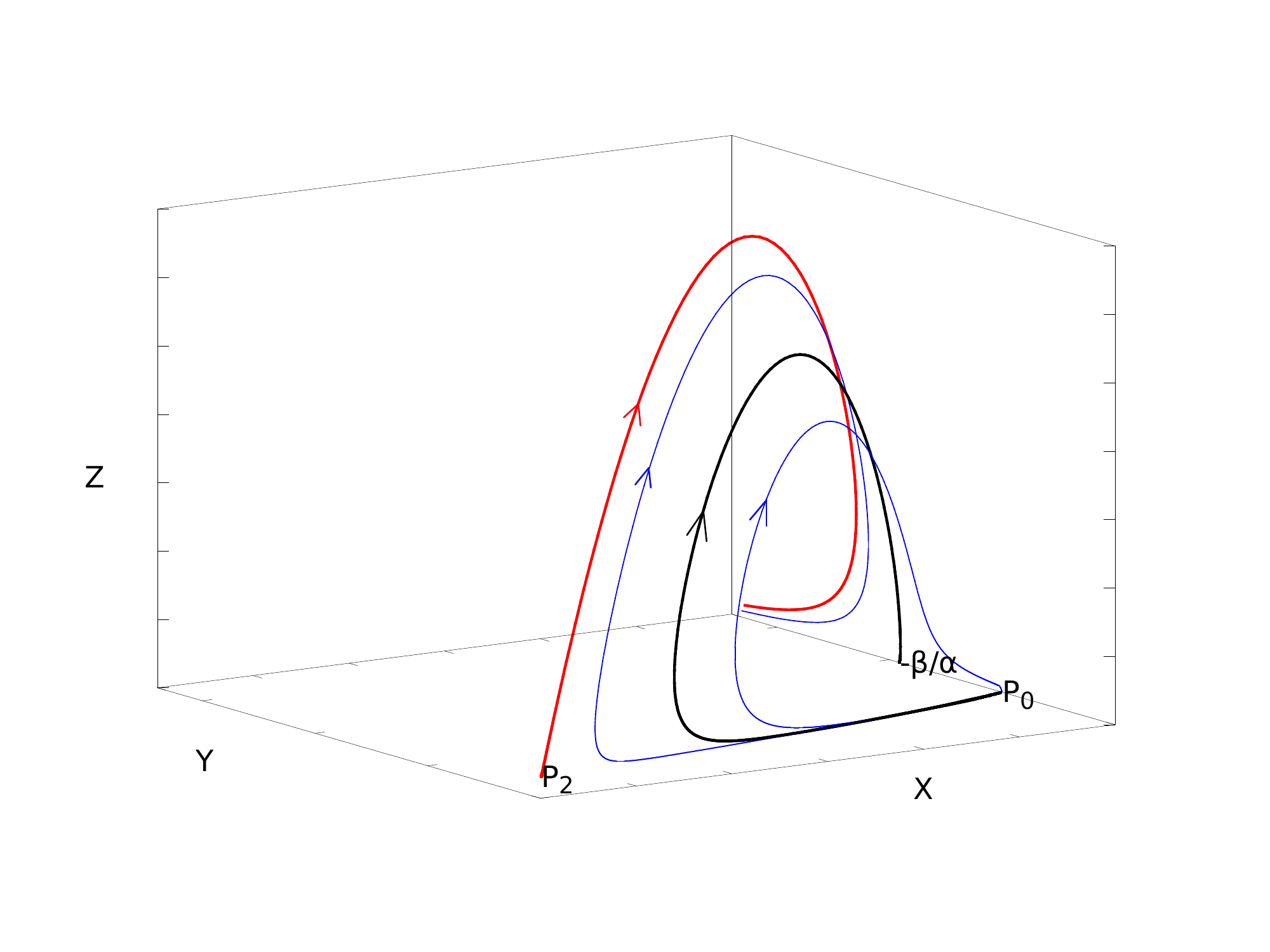}}
  \end{center}
  \caption{Orbits from $P_2$ and $P_0$ for different values of $\sigma$. Numerical experiment for $m=3$, $p=0.5$ and $\sigma=3.233$, respectively $\sigma=3.5$}\label{fig3}
\end{figure}

\section{Non-existence when $m+p<2$}\label{sec.non}

This section is devoted to the case $m+p<2$ and the proof of Theorem \ref{th.non}. Let us notice first, at a formal level, that the previous study gives us the idea that good blow-up profiles with interface do not exist. As we know already, in the phase space associated to the system \eqref{PSsyst1} for $m+p>2$ the critical point encoding the interface behavior are $P_0$ and $P_1$. An inspection of the analysis in Section \ref{sec.type2} for $P_0$ and in Lemma \ref{lem.P1} for the point $P_1$ gives that, if the expression $m+p-2$ changes sign, big differences occur. Indeed, recalling the invariants $D$, $K_2$ and $K_3$ in the analysis in \cite{Date79}, we notice that when $m+p-2<0$ we are in the case $D<0$, $K_2>0$, $K_3<0$ which corresponds to the phase portrait number 3 in \cite[Figure 8, p. 329]{Date79}, showing that there are no longer orbits entering $P_0$. On the other hand, the analysis in Lemma \ref{lem.P1} changes as the third eigenvalue $\lambda_3=-(m+p-2)\beta/\alpha$ becomes positive, thus by inspecting the eigenvectors for $P_1$ we obtain that there are no profiles either entering or going out of $P_1$. Indeed, keeping the analysis in Lemma \ref{lem.P1}, the linearization near the point would have a two-dimensional unstable manifold generated by the eigenvectors
$$
e_2=(0,1,0), \ e_3=\left(0,1,\frac{(m+p-1)\beta}{\alpha}\right),
$$
and included completely in the invariant plane $\{X=0\}$ and a one-dimensional stable manifold contained in the invariant plane $\{Z=0\}$, none of them containing solutions to \eqref{SSODE}. 

All the above are of course formal considerations, as these analysis do not remain valid when $m+p<2$ since in this case $Z(\xi)=\infty$ and thus the critical points $P_0$ and $P_1$ do not exist anymore in the same form as we studied them. But still, these formal arguments give us an understanding of why interfaces disappear when $m+p<2$. To make it rigorous, it is sufficient to introduce a phase space for a system where the critical points can be analyzed one by one and show that no interface behavior may exist. Unfortunately, the system we used \eqref{PSsyst1} is not good for this aim since the points $P_0$, $P_1$ and $P_2$ will all unify with the point $Q_4$ at infinity making the analysis very difficult. We thus have to introduce a new quadratic autonomous system where the component $Z$ behaves well. We are led to the following change of variable:
\begin{equation}\label{PSvar5}
\begin{split}
&X(\eta)=\sqrt{m}\xi^{-(\sigma+2)/2}f^{(m-p)/2}(\xi), \ Y(\eta)=\sqrt{m}\xi^{-\sigma/2}f^{(m-p-2)/2}(\xi)f'(\xi), \\ &Z(\eta)=\frac{\alpha}{\sqrt{m}}\xi^{(2-\sigma)/2}f^{(2-m-p)/2}(\xi),
\end{split}
\end{equation}
where we recall that $\alpha$ is defined in \eqref{SSexp} and the new independent variable $\eta=\eta(\xi)$ is defined through the differential equation
$$
\frac{d}{d\eta}=\sqrt{m}\xi^{-\sigma/2}f^{m-p}(\xi)\frac{d}{d\xi}.
$$
The differential equation \eqref{SSODE} transforms into the system
\begin{equation}\label{PSsyst5}
\left\{\begin{array}{ll}\dot{X}=X\left[\frac{m-p}{2}Y-\frac{\sigma+2}{2}X\right],\\
\dot{Y}=-\frac{m+p}{2}Y^2-\frac{\sigma}{2}XY+XZ-\frac{\beta}{\alpha}ZY-1,\\
\dot{Z}=Z\left[\frac{2-m-p}{2}Y-\frac{2-\sigma}{2}X\right].\end{array}\right.
\end{equation}
and it is easy to see that \eqref{PSsyst5} does not have finite critical points, thus all its critical points are at infinity. The most important favorable thing related to the system \eqref{PSsyst5} is that, since for $m+p<2$ we have
$$
\sigma>\frac{2(1-p)}{m-1}>2
$$
and the definition of $Z$ in \eqref{PSvar5}, any \textbf{interface behavior} or even tail behavior as $\xi\to\infty$ has to be seen in a critical point at infinity with \textbf{the component $Z=0$}. It is thus sufficient to study the critical points at infinity for the system \eqref{PSsyst5} to show that such behavior is impossible and prove Theorem \ref{th.non}. We will be rather brief below, skipping some technical details as the analysis is very similar to the one performed in Section \ref{sec.local}.

\medskip

\noindent \begin{proof}[Proof of Theorem \ref{th.non}] We pass to the Poincar\'e hypersphere following \cite[Section 3.10]{Pe}. We introduce new variables
$(\overline{X},\overline{Y},\overline{Z},W)$ by letting:
$$
X=\frac{\overline{X}}{W}, \ \ Y=\frac{\overline{Y}}{W}, \ \ Z=\frac{\overline{Z}}{W}
$$
and according to \cite[Theorem 4, Section 3.10]{Pe}, the critical points at infinity in the phase space associated to the system \eqref{PSsyst5} lie on the Poincar\'e hypersphere at points $(\overline{X},\overline{Y},\overline{Z},0)$ where $\overline{X}^2+\overline{Y}^2+\overline{Z}^2=1$ and they solve the system
\begin{equation}\label{Poincare11}
\left\{\begin{array}{ll}\overline{X}Q_2(\overline{X},\overline{Y},\overline{Z})-\overline{Y}P_2(\overline{X},\overline{Y},\overline{Z})=0,\\
\overline{X}R_2(\overline{X},\overline{Y},\overline{Z})-\overline{Z}P_2(\overline{X},\overline{Y},\overline{Z})=0,\\
\overline{Y}R_2(\overline{X},\overline{Y},\overline{Z})-\overline{Z}Q_2(\overline{X},\overline{Y},\overline{Z})=0,\end{array}\right.
\end{equation}
where $P_2$, $Q_2$ and $R_2$ are the homogeneous second degree parts of the polynomials in the right hand side of the system \eqref{PSsyst5}, that is
\begin{equation*}
\begin{split}
&P_2(\overline{X},\overline{Y},\overline{Z})=\overline{X}\left[\frac{m-p}{2}\overline{Y}-\frac{\sigma+2}{2}\overline{X}\right],\\
&Q_2(\overline{X},\overline{Y},\overline{Z})=-\frac{m+p}{2}\overline{Y}^2-\frac{\sigma}{2}\overline{X}\overline{Y}+\overline{X}\overline{Z}-\frac{\beta}{\alpha}\overline{Y}\overline{Z},\\
&R_2(\overline{X},\overline{Y},\overline{Z})=\overline{Z}\left[\frac{2-m-p}{2}\overline{Y}+\frac{2-\sigma}{2}\overline{X}\right].
\end{split}
\end{equation*}
The system \eqref{Poincare11} becomes
\begin{equation}\label{Poincare12}
\left\{\begin{array}{ll}\overline{X}\left(-m\overline{Y}^2+\overline{X}\overline{Y}+\overline{X}\overline{Z}-\frac{\beta}{\alpha}\overline{Y}\overline{Z}\right)=0,\\
\overline{X}\overline{Z}(2\overline{X}-(m-1)\overline{Y})=0,\\
\overline{Z}\left(\overline{Y}^2+\overline{X}\overline{Y}-\overline{X}\overline{Z}+\frac{\beta}{\alpha}\overline{Y}\overline{Z}\right)=0,\end{array}\right.
\end{equation}
Straightforward calculations give that the system \eqref{Poincare12} has seven critical points and each one of them has a direct correspondence to the critical points studied in Section \ref{sec.local}. We give (in a rather sketchy way) their analysis one by one below.

\medskip

$\bullet$ The critical point $(1,0,0,0)$ in the Poincar\'e hypersphere is topologically equivalent, according to part (a) of \cite[Theorem 5, Section 3.10]{Pe}, to the origin in the system:
\begin{equation}\label{systinf11}
\left\{\begin{array}{ll}-\dot{y}=-y-z+w^2+my^2+\frac{\beta}{\alpha}yz,\\
-\dot{z}=-2z+(m-1)yz,\\
-\dot{w}=-\frac{\sigma+2}{2}w+\frac{m-p}{2}yw,\end{array}\right.
\end{equation}
where the minus sign has been chosen according to the direction of the flow in the original system \eqref{PSsyst5}. It readily follows that this point is an unstable node containing profiles such that
$$
(f(\xi)^{(m-p)/(\sigma+2)})'\sim K \qquad {\rm as} \ \xi\to0,
$$
thus gathering in it the profiles corresponding to the old points $Q_1$ and $P_0$ in the system \eqref{PSsyst1}.

$\bullet$ The critical points $(0,\pm1,0,0)$ in the Poincar\'e hypersphere are topologically equivalent, according to part (b) of \cite[Theorem 5, Section 3.10]{Pe}, to the origin in the system:
\begin{equation}\label{systinf21}
\left\{\begin{array}{ll}\pm\dot{x}=-mx-\frac{\beta}{\alpha}xz+x^2-xw^2+x^2z,\\
\pm\dot{z}=-z-xz-\frac{\beta}{\alpha}z^2-zw^2+xz^2,\\
\pm\dot{w}=-\frac{m+p}{2}w-\frac{\beta}{\alpha}zw-\frac{\sigma}{2}xw-w^3+xzw,\end{array}\right.
\end{equation}
where the minus sign corresponds to one of the points and the plus sign to the other one. We see that both these points are nodes (one unstable and one stable) and the orbits connecting to them contain profiles with a change of sign at some positive point $\xi_0\in(0,\infty)$. These points correspond to the critical points $Q_2$ and $Q_3$ in our initial system \eqref{PSsyst1}.

$\bullet$ The critical point $(0,0,1,0)$ cannot contain profiles having either an interface or a tail behavior as $\xi\to\infty$, since we noticed that this implies $\overline{Z}=0$. In fact, this point corresponds to the critical point $Q_4$ in the system \eqref{PSsyst1}.

$\bullet$ The critical point
$$
\left(\frac{m}{\sqrt{1+m^2}},\frac{1}{\sqrt{1+m^2}},0,0\right)
$$
is topologically equivalent to the critical point $(y,z,w)=(1/m,0,0)$ in the system \eqref{systinf11}. It is straightforward (by imitating the proof of Lemma \ref{lem.Q5}) that this point contains profiles with a change of sign at $\xi=0$ of the form $f(\xi)\sim K\xi^{1/m}$ as $\xi\to0$ and it corresponds to the critical point $Q_5$ in the system \eqref{PSsyst1}.

$\bullet$ The critical point
$$
\left(0,-\frac{\beta}{\sqrt{\alpha^2+\beta^2}},\frac{\alpha}{\sqrt{\alpha^2+\beta^2}},0\right)
$$
is topologically equivalent to the critical point $(0,-\alpha/\beta,0)$ in the system \eqref{systinf21}, and the linearization of the system \eqref{systinf21} in a neighborhood of it has the matrix
$$
M=\left(
             \begin{array}{ccc}
               1-m & 0 & 0 \\
               \frac{\alpha(\alpha+\beta)}{\beta^2} & 1 & 0 \\
               0 & 0 & \frac{2-m-p}{2} \\
             \end{array}
           \right),
$$
with eigenvalues $\lambda_1=1-m<0$, $\lambda_2=1$ and $\lambda_3=(2-m-p)/2>0$ and eigenvectors
$$
e_1=\left(-\frac{m\beta^2}{\alpha(\alpha+\beta)},1,0\right), \ e_2=(0,1,0), \ e_3=(0,0,1).
$$
A standard analysis shows that the two-dimensional unstable manifold lies in the invariant plane $\{x=0\}$ while the one-dimensional stable manifold lies in the invariant plane $\{w=0\}$ (both invariant planes corresponding to the system \eqref{systinf21}), and the orbits contained in these manifolds contain no profiles. This is the point that would have been the correspondent to the critical point $P_1$ in the system \eqref{PSsyst1}, as explained in the formal considerations related to the change of sign of $m+p-2$ at the beginning of the current Section.

$\bullet$ There exists one more critical point having all three non-zero components, that is
$$
\left(\frac{m-1}{2L},\frac{1}{L},\frac{\alpha(m+1)}{L}\right), \qquad L^2=1+(m+1)^2\alpha^2+\frac{(m-1)^2}{4}.
$$
A detailed analysis of this point shows that it corresponds to the critical point $P_2$ in the system \eqref{PSsyst1}. However, for our goals the point can be discarded even without performing this analysis, as we explained that the points of interest for the interface or tail behavior should necessarily have $\overline{Z}=0$.

\medskip

Since these are all the critical points in the system \eqref{PSsyst5}, and they codify thus all the information about the blow-up profiles, we conclude that there is no blow-up profile either with interface at a finite $\xi_0\in(0,\infty)$ or with a tail behavior as $\xi\to\infty$, ending the proof.
\end{proof}

\section*{Final comments and extensions}

We gather in this final page some comments about the remaining cases and some open problems.

\medskip

\noindent \textbf{1. The very interesting case $m+p=2$} is not considered in the current work and is studied in the companion paper \cite{IS20crit}. This is because a significant number of differences in the techniques appear. By inspecting for example the autonomous system \eqref{PSsyst1} we notice that the equation for $\dot{Z}$ simplifies and instead of the critical points studied here, we have \emph{a critical parabola} of equation
$$
-Y^2-\frac{\beta}{\alpha}Y-Z=0,
$$
formed by critical points and connecting the critical points $P_0=(0,0,0)$ and $P_1=(0,-\beta/\alpha,0)$. Studying the parabola involves different techniques than the ones we have used in the present work. Moreover, an interesting feature of the critical case $m+p=2$ is that the interface behaviors coincide and \emph{there cannot be an interface at some $\xi_0\in(0,\infty)$ sufficiently large}, radically contrasting the results in Section \ref{sec.type1}. The backward shooting method is no longer applicable for this case and it will be replaced by other techniques based on the geometry of the phase space.

\medskip

\noindent \textbf{2. The uniqueness of good profiles with interface of Type I} is an interesting open problem to be raised in relation to the results we get in the present paper. Indeed, good profiles with interface are not unique and moreover we show in Section \ref{sec.type2} that good profiles with interface of Type II are infinite for any fixed $\sigma$. However, the local uniqueness of the Type I interface behavior at a given $\xi_0\in(0,\infty)$ (see Proposition \ref{prop.uniq}) and the proof of Theorem \ref{th.exist} by the backward shooting method suggest at an intuitive level that the uniqueness of this type of profile for a given $\sigma$ is expected to be true. We do not have any clue about a proof of it and we feel that it requires to obtain results of monotonicity (of some of the trajectories at least, or of the global change of the phase space) with respect to $\sigma$, a task which is usually very difficult.

\medskip

\noindent \textbf{3. The non-existence of profiles when $m+p<2$} given as Theorem \ref{th.non} hides in fact a deeper fact: it is expected that \emph{no solution except the zero one} exists for Eq. \eqref{eq1} when $m+p<2$. In the related paper \cite{IMS20} we show, among other results, such a sharp non-existence result for a related equation with a stronger weight on the reaction term, that is
$$
\partial_tu=\Delta u^m+(1+|x|)^{\sigma}u^p
$$
in the same range $m+p<2$. But we feel that the difference between $(1+|x|)^{\sigma}$ and $|x|^{\sigma}$ is not so essential for the non-existence and our Theorem \ref{th.non} confirms these expectations.

\bigskip

\noindent \textbf{Acknowledgements} A. S. is partially supported by the Spanish project MTM2017-87596-P.

\bibliographystyle{plain}

\end{document}